\documentclass[12pt]{article}
\usepackage {graphics} 
\usepackage {graphicx}  
\usepackage {pstricks,pst-node,  pst-coil} 
\onecolumn  
\usepackage{amsmath,amssymb,latexsym}
\usepackage{amsmath,amsthm}
\usepackage{enumitem}
\usepackage{mathtools}

\usepackage{tikz}

\pgfdeclarelayer{edgelayer}
\pgfdeclarelayer{nodelayer}
\pgfsetlayers{edgelayer,nodelayer,main}

\tikzstyle{black_dot}=[fill=black, draw=black, shape=circle, inner sep=0pt, minimum size=3pt]
\tikzstyle{white_circle}=[fill=white, draw=black, shape=circle, inner sep=0pt, minimum size=3pt]
\tikzstyle{empty}=[fill=none, draw=none, shape=rectangle, minimum size=0pt, inner sep=0pt]

\tikzstyle{dotted_edge}=[-, dashed]
\tikzstyle{blue_edge}=[-, draw=blue]
\tikzstyle{red_edge}=[-, draw=red]
\tikzstyle{green_edge}=[-, draw=green]
\tikzstyle{yellow_edge}=[-, draw=yellow]
\tikzstyle{right_arrow}=[fill=none, ->]

\newtheorem{tw}{Theorem}  
\newtheorem{lem}[tw]{Lemma}  
\newtheorem{problem}[tw]{Problem}  
  
\newtheorem{cnj}[tw]{Conjecture}  
\newtheorem{cor}[tw]{Corollary}  
\newtheorem{rem}[tw]{Remark}  
\newtheorem{claim}[tw]{Claim}
\newtheorem{observation}[tw]{Observation}

\title{On a new problem about the local irregularity
of graphs}
\author{Igor Grzelec\thanks{Department of Discrete Mathematics, AGH University of Krakow, Poland.}, Tomáš Madaras\thanks{Institute of mathematics, P.J. Šafárik University, Košice, Slovakia}\thanks{The corresponding author. Email: tomas.madaras@upjs.sk}, Alfréd Onderko\footnotemark[2], Roman Soták\footnotemark[2]}

\date{October 2024}

\begin{document}

\maketitle
\begin{abstract}
A graph/multigraph $G$ is \textit{locally irregular} if endvertices of every its edge possess different degrees. The \textit{locally irregular edge coloring} of $G$ is its edge coloring with the property that every color induces a locally irregular sub(multi)graph of $G$; if such a coloring of $G$ exists, the minimum number of colors to color $G$ in this way is the \textit{locally irregular chromatic index} of $G$ (denoted by ${\rm lir}(G)$). We state the following new problem: given a connected graph $G$ distinct from $K_2$ or $K_3$, what is the minimum number of edges of $G$ to be doubled such that the resulting  multigraph is locally irregular edge colorable (with no monochromatic multiedges) using at most two colors? This problem is closely related to several open conjectures (like the Local Irregularity Conjecture for graphs and 2-multigraphs, or (2, 2)-Conjecture) and other similar edge coloring concepts. We present the solution of this problem for several graph classes: paths, cycles, trees, complete graphs, complete $k$-partite graphs, split graphs and powers of cycles. Our solution for complete $k$-partite graphs ($k>1$) and powers of cycles (which are not complete graphs) shows that, in this case, the locally irregular chromatic index equals 2. We also consider this problem for special families of cacti and prove that the minimum number of edges in a graph whose doubling yields an local irregularly colorable  multigraph does not have a constant upper bound not only for locally irregular uncolorable cacti. \\

\textit{Keywords:} Locally irregular graphs  and multigraphs; Locally irregular edge coloring; Powers of cycles.
\end{abstract}

\section{Introduction}
Throughout this paper, we consider finite graphs and multigraphs. 
We say that a (multi)graph $G$ is \textit{locally irregular} if, for every edge of $G$, its endvertices have different degrees. 
An edge coloring of $G$ such that each color class induces a locally irregular subgraph of $G$ is called a \textit{locally irregular coloring}; the one that uses $k$ colors is called locally irregular $k$-edge coloring ( $k$-liec in short).
By $\mathrm{lir}(G)$, we denote the \textit{locally irregular chromatic index} of $G$, which is defined as the minimum number $k$ such that there exists a $k$-liec of $G$. 

The problem of determining the value of $\mathrm{lir}(G)$ has close connection to the well-known 1-2-3 Conjecture proposed in \cite{Karonski Luczak Thomason} by Karoński, Łuczak and Thomason (it was recently proved by Keusch in \cite{Keusch}). While it states that it is enough to multiply some edges of a graph at most three times to transform it into a locally irregular multigraph, one can observe (\cite{Baudon Bensmail Przybylo Wozniak}, Observation 1.3) that, for a regular graph $G$, the existence of such a transform involving just doubling the edges is equivalent to the existence of a 2-liec of $G$.  

Note, however, that not every graph can be locally irregularly colorable. To characterize completely the family of all uncolorable graphs, let us define recursively the family $\mathfrak{T}$ in the following way:
\begin{itemize}
  \item The cycle $C_3$ is the member of  $\mathfrak{T}$.
  \item Let $G$ be a graph from $\mathfrak{T}$. Let $G'$ be a graph obtained from $G$ by identifying a vertex $v\in V(G)$ of degree two lying on a $3$-cycle in $G$, with an endvertex of an even-length path or with an endvertex of an odd-length path such that the other endvertex of that path is identified with a vertex of a new $3$-cycle. Then $G'$ belongs to $\mathfrak{T}$.
\end{itemize}
Take now the family $\mathfrak{T'}$ consisting of all graphs from $\mathfrak{T}$, all odd-length paths, and all odd-length cycles. Baudon, Bensmail, Przybyło, and Woźniak showed in \cite{Baudon Bensmail Przybylo Wozniak} that the only locally irregular uncolorable connected graphs are those ones from $\mathfrak{T'}$.
In addition, they also proposed the Local Irregularity Conjecture which states that every connected graph $G \notin \mathfrak{T'}$ satisfies $\mathrm{lir}(G)\leq 3$. 

Since $\mathfrak{T}'$ is a subclass of the family of cacti (that is, connected graphs in which no two cycles share more than one vertex), determining the value of $\mathrm{lir}(G)$ for cacti seemed interesting in particular.
When studying locally irregular colorings of sparse graphs (cacti among them), in 2021 Sedlar and \v Skrekovski~\cite{Sedlar Skrekovski} disproved the original Local Irregularity Conjecture 
by showing that the bow-tie graph $B$ (see Figure \ref{bow-tie graph}) does not have a locally irregular coloring with fewer than four colors.
\begin{figure}[h!]
\centering
    \begin{tikzpicture}
    \begin{pgfonlayer}{nodelayer}
    \node [style={black_dot}] (0) at (-0.5, 0) {};
    \node [style={black_dot}] (1) at (0.5, 0) {};
    \node [style={black_dot}] (2) at (-0.5, 1) {};
    \node [style={black_dot}] (3) at (-1.5, 0.5) {};
    \node [style={black_dot}] (4) at (-1.5, -0.5) {};
    \node [style={black_dot}] (5) at (-0.5, -1) {};
    \node [style={black_dot}] (6) at (0.5, -1) {};
    \node [style={black_dot}] (7) at (0.5, 1) {};
    \node [style={black_dot}] (8) at (1.5, 0.5) {};
    \node [style={black_dot}] (9) at (1.5, -0.5) {};
    \end{pgfonlayer}
    \begin{pgfonlayer}{edgelayer}
    \draw [style={red_edge}] (2) to (0);
    \draw [style={red_edge}] (0) to (1);
    \draw [style={red_edge}] (0) to (5);
    \draw [style={blue_edge}] (5) to (4);
    \draw [style={blue_edge}] (4) to (0);
    \draw [style={blue_edge}] (1) to (9);
    \draw [style={blue_edge}] (9) to (6);
    \draw [style={green_edge}] (0) to (3);
    \draw [style={green_edge}] (3) to (2);
    \draw [style={green_edge}] (7) to (8);
    \draw [style={green_edge}] (8) to (1);
    \draw [style={yellow_edge}] (7) to (1);
    \draw [style={yellow_edge}] (1) to (6);
    \end{pgfonlayer}
    \end{tikzpicture}
\caption{Locally irregular 4-coloring of the bow-tie graph $B$.}
\label{bow-tie graph}
\end{figure}
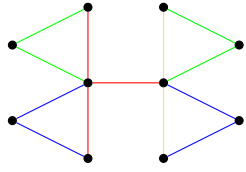
In \cite{Sedlar Skrekovski 2}, Sedlar and \v Skrekovski proved that every locally irregular colorable cactus $G \neq B$ satisfies $\mathrm{lir}(G)\leq 3$. 
Furthermore, they established the following improved version of the Local Irregularity Conjecture.
\begin{cnj}[Local Irregularity Conjecture~\cite{Sedlar Skrekovski 2}]
\label{graph3}
    Let $G \neq B$ be a connected locally irregular colorable graph. Then $\mathrm{lir}(G)\leq 3$.
\end{cnj}

Conjecture~\ref{graph3} was confirmed for graphs from specific classes like trees (see~\cite{Baudon Bensmail}), graphs (general or regular) with high minimum degree (see~\cite{Przybylo} and~\cite{Baudon Bensmail Przybylo Wozniak}), locally irregularly colorable split graphs and subcubic claw-free graphs (see~\cite{Lintzmayer Mota Sambinelli} and~\cite{Luzar Macekova}). 
Apart from the above-mentioned results, it was proved in~\cite{Bensmail Dross Nisse} that $\mathrm{lir}(G)\leq 15$ if $G$ is planar. 
For all connected graphs, Bensmail, Merker and Thomassen~\cite{Bensmail Merker Thomassen} showed that $\mathrm{lir}(G)\leq 328$ if $G \notin \mathfrak{T'}$;  
Lu\v zar, Przybyło and Soták~\cite{Luzar Przybylo Sotak} later decreased this bound 220.

For 2-multigraphs (that is, multigraphs in which the multiplicity of every edge is exactly two; the one resulting in this way from a simple graph $G$ will be denoted by $^2G$ in the sequel),
Grzelec and Woźniak established the following conjecture:
\begin{cnj}[Local Irregularity Conjecture for 2-multigraphs~\cite{Grzelec wozniak}]
\label{main}
    Let $G$ be a connected graph distinct from $K_2$. Then ${\rm lir}(^2G)\leq 2$.
\end{cnj}

This conjecture was confirmed in~\cite{Grzelec wozniak} for graphs of several particular families.
Moreover, in the same article~\cite{Grzelec wozniak}, a general constant upper bound of 76 was proved for $\mathrm{lir}(^2G)$.
Then, in~\cite{Grzelec wozniak2}, Grzelec and Woźniak proved Conjecture~\ref{main} for cacti. 

Another approach to the local irregularity of multigraphs was introduced in~\cite{Baudon Bensmail Davot Hocquard} by Baudon et al. It combines neighbor-sum-distinguishing edge coloring and graph decomposition into locally irregular graphs. In the neighbor-sum-distinguishing edge coloring, two vertices $x$ and $y$ are \textit{distinguished} if $\sigma (x)\neq \sigma (y)$, where $\sigma (x):=\sum \limits_{x \in e}c(e)$, $\sigma (y):=\sum \limits_{y \in e}c(e)$ and $c$ is an edge coloring using colors from $\{1, 2, \dots, k\}$. 
Such a coloring can be seen as a transformation of $G$ to a multigraph such that each edge $e$ of $G$ is replaced by $c(e)$ parallel edges.
The following conjecture was proposed:  
\begin{cnj}[$(2, 2)$-Conjecture~\cite{Baudon Bensmail Davot Hocquard}]
\label{2,2n}
Every connected graph $G$ of order at least four can be decomposed into two subgraphs $G_1$ and $G_2$ such that there exist locally irregular multigraphs $M_1$ and $M_2$ for which
$G_1 \subseteq M_1 \subseteq ^2G_1$ and $G_2 \subseteq M_2 \subseteq ^2G_2$.
\end{cnj}
Conjecture~\ref{2,2n} can be also formulated in the language of $(p, q)$-colorings (decompositions of graphs into at most $p$ subgraphs such that, in each of these subgraphs, the neighboring vertices can be distinguished by sums using at most $q$ colors): 
\begin{cnj}[$(2, 2)$-Conjecture~\cite{Baudon Bensmail Davot Hocquard}]
\label{2,2}
    Every connected graph of order at least four has a $(2,2)$-coloring.
\end{cnj}

The (2, 2)-Conjecture was confirmed for some classes of connected graphs of order at least four, such as locally irregular uncolorable graphs \cite{Baudon Bensmail Davot Hocquard}, complete graphs \cite{Baudon Bensmail Davot Hocquard}, bipartite graphs \cite{Baudon Bensmail Davot Hocquard}, 2-degenerate graphs \cite{Baudon Bensmail Davot Hocquard}, subcubic graphs \cite{Baudon Bensmail Davot Hocquard} and graphs of minimum degree at least $10^6$ \cite{Przybylo2}.

Motivated by the above-mentioned problems, we define a new one, closely related to previously-mentioned conjectures.
\begin{problem}
\label{main_problem}
Let $G$ be a connected graph that is not isomorphic to $K_2$ or $K_3$. 
What is the minimum number of edges of $G$ whose doubling yields a multigraph which is locally irregular edge colorable using at most two colors with no multiedges colored with two colors?
\end{problem}

In this paper, we deal only with Problem~\ref{main_problem}, but a slightly different problem may be also stated. We present it and give some notes on it in the last section.

If $G$ is a graph and $E_d$ is a subset of its edges, by $G + E_d$ we denote the multigraph obtained from $G$ by replacing each edge from $E_d$ with two parallel edges.

In most cases, we are interested in the construction of a coloring that uses at most two colors.
Hence, for convenience, we use red and blue to represent those two colors. 
The number of edges incident to a vertex $v$ that are colored red in a coloring of a (multi)graph $G$ is called the red degree of $v$ and denoted by $\deg_R(v)$.
We use the analogous definition for the blue degree $\deg_B(v)$ of $v$.
The notion of $\deg_R(v)$ and $\deg_B(v)$ reflects the degree in a classical sense of $v$ in the red subgraph $R$ of $G$ and the blue subgraph $B$ of $G$.

By $\mathcal{D}_{\mathrm{lir}}(G)$ we denote the minimum number of doubled edges required in the multigraph $M$ obtained from a graph $G$ such that $M$ has a locally irregular edge coloring with at most two colors and without two-colored multiedges. 
Note that Problem~\ref{main_problem} has a solution for a connected graph $G$  which is not isomorphic to $K_2$ or $K_3$ if and only if $(2, 2)$-Conjecture holds for $G$. 
Moreover, if a graph $G$ satisfies $\mathrm{lir}(G)\leq 2$ then $\mathcal{D}_{\mathrm{lir}}(G)=0$. 
In this paper, we will show a solution to Problem \ref{main_problem} for paths, cycles, trees, complete graphs, split graphs, and special cacti similar to those in $\mathfrak{T}$. 
On top of that, we show that $\mathrm{lir}(G) = 2$ whenever $G$ is a complete $k$-partite graph or a power of a cycle different from a complete graph, which yields $\mathcal{D}_\mathrm{lir}(G) = 0$ in these cases.  

\section{Paths, cycles, trees and special cacti}

In~\cite{Baudon Bensmail Przybylo Wozniak} it was shown that the path on $n$ vertices $P_n$ admits a locally irregular coloring if and only if $n$ is odd, and that two colors are enough in those cases.
The situation is similar for cycles, where odd cycles do not admit locally irregular colorings with any number of colors, while cycles of length divisible by four admit a 2-liec, and those of length congruent to two modulo four admit a 3-liec (but no 2-liec).
The deciding factor for locally irregular edge colorability in the case of paths and cycles is, not surprisingly, whether they can be partitioned into disjoint paths of length two, as $P_3$ is the only instance of a locally irregular graph among cycles and paths of positive length.

Hence, to solve the problem of determining $\mathcal{D}_\mathrm{lir}(G)$ in the case when $G$ is a path or a cycle, it is enough to consider paths of odd length, and cycles of length not divisible by four, as $\mathcal{D}_\mathrm{lir}(G) = 0$ is implied by $\mathrm{lir}(G) \leq 2$ in the remaining cases.
As was mentioned, in the class of simple graphs, among paths and cycles of positive length, there is only one instance of locally irregular graph, namely $P_3$.
When considering multigraphs (due to the possibility of doubling some edges to obtain a 2-locally irregular colorable multigraph), the following observation is useful (in the case of $P_4$ double a pendant edge, in the case of $P_5$ and $P_7$ double both central edges, and in the case of $P_6$ double the central edge and an edge adjacent to it):
\begin{observation}\label{obs_mochromatic_paths}
    It is necessary to double an edge of $P_4$ to obtain a locally irregular multigraph, and 
    it is necessary to double two edges of $P_k$, $k \in \{5,6,7\}$ to obtain a locally irregular multigraph.
\end{observation}
Since $P_3$ is locally irregular itself, and doubling of the edge $P_2$ results in a regular multigraph, we get the following:
\begin{tw}\label{paths}
    If $n \geq 3$ then 
    \begin{align*}
        \mathcal{D}_{\mathrm{lir}}(P_n) = \left\{\begin{matrix}
         0 & \text{if } n \equiv 1 \pmod{2},\\
         1 &\text{if } n \equiv 0 \pmod{2}.
    \end{matrix}\right.
    \end{align*}
\end{tw}
\begin{proof}
    $\mathcal{D}_\mathrm{lir}(P_n)=0$ is implied by $\mathrm{lir}(P_n)=2$ in the case of odd $n$.
    If $n$ is even, partition $P_n$ into edge-disjoint $\frac{n-4}{2}$ paths of length two and a path of length three, and use a doubling on $P_4$.
    Assign colors to the paths in the partition in such a way that two paths sharing a vertex receive distinct colors and two vertex-disjoint paths are colored the same.
\end{proof}

Observation~\ref{obs_mochromatic_paths} and the fact that $P_3$ is locally irregular can be also used to determine the value of $\mathcal{D}_\mathrm{lir}(C_n)$.
\begin{tw}\label{cycles}
    If $C_n$ is a cycle of length $n \geq 4$ then
    \begin{align*}
        \mathcal{D}_{\mathrm{lir}}(C_n) = \left\{\begin{matrix}
             0 & \text{if } n \equiv 0 \pmod{4},\\
             1 &\text{if } n \equiv 1 \pmod{4},\\
             2 &\text{otherwise.}
        \end{matrix}\right.
    \end{align*}
\end{tw}
\begin{proof}
    If $n \equiv 0 \pmod{4}$ then $\mathrm{lir}(C_n) = 2$, and consequently $\mathcal{D}_\mathrm{lir}(C_n) = 0$. 
    
    If $n \equiv 1 \pmod{4}$ then we can partition $C_n$ into edge-disjoint path $P_4$ and $\frac{n-3}{2}$ paths $P_3$. 
    We double a pendant edge on $P_4$ and color this path red. 
    The remaining paths we color alternately blue and red (each path in one color). 
    Therefore, in this case $\mathcal{D}_\mathrm{lir}(C_n) = 1$, because the cycle has odd length and one doubling is necessary.

    If $n \equiv 2 \pmod{4}$ then we can partition $C_n$ into edge-disjoint path $P_5$ and $\frac{n-4}{2}$ paths $P_3$.
    We double first two edges on path $P_5$ and color this path red. 
    The remaining paths we color alternately blue and red. 
    Therefore, in this case $\mathcal{D}_\mathrm{lir}(C_n) = 2$, because $\mathrm{lir}(C_n) = 3$ and one doubling is not enough from Observation~\ref{obs_mochromatic_paths}.

    If $n \equiv 3 \pmod{4}$ then we can partition $C_n$ into edge-disjoint path $P_6$ and $\frac{n-5}{2}$ paths $P_2$. 
    We double the second and the third edge on $P_6$ and color this path red. 
    The remaining paths we color alternately blue and red. 
    Therefore, in this case $\mathcal{D}_\mathrm{lir}(C_n) = 2$, because the cycle has odd length and one doubling is not enough from Observation~\ref{obs_mochromatic_paths}.
\end{proof}

In~\cite{Baudon Bensmail}, authors proved that there is a linear-time algorithm that decides if $\mathrm{lir}(T) \leq 2$ for a tree $T$; this complemented the previously known results that every tree $T$ different from an odd-length path has $\mathrm{lir}(T) \leq 3$ (see~\cite{Baudon Bensmail Przybylo Wozniak}) and that $\mathrm{lir}(T) \leq 2$ if the maximum degree of a tree $T$ is at least five (see~\cite{Baudon Bensmail Sopena}).
As a byproduct of proving the linearity of determining the existence of 2-liec in the case of trees, Baudon, Bensmail and Sopena proved that every shrub admits a 2-aliec (almost 2-locally irregular edge coloring, i.e., a coloring in which every connected monochromatic component is locally irregular except at most one, which is then the single root edge):
\begin{tw}[Baudon, Bensmail, Sopena~\cite{Baudon Bensmail}]\label{aliec_trees}
    Every shrub admits a $2$-aliec.
\end{tw}
The proof of the following theorem uses the above-mentioned results, and, in particular, deals with the remaining cases when there is no 2-liec of a tree. 

\begin{tw}
    If $T$ is a tree of order at least three then $\mathcal{D}_\mathrm{lir}(G) \leq 1$.
\end{tw}
\begin{proof}
    Let $T$ be a tree of order at least three without a 2-liec, as otherwise $\mathcal{D}_\mathrm{lir}(T) = 0$.
    We may suppose that $T$ is a shrub; take a pendant vertex $x$ of $T$ as a root.
    Then $\Delta(T) \leq 4$ (see Theorem~3 from~\cite{Baudon Bensmail Sopena}).
    Let $u$ be a neighbor of $x$ in $T$.
    From Theorem~\ref{aliec_trees} we have that there is a 2-aliec of $T$, in which $ux$ is colored red, and all other edges incident to $u$ are blue.
    We will distinguish several cases depending on $\deg_T(u)$.

    \textbf{Case 1.} Let $\deg_T(u) = 2$.
    Denote by $v$ the neighbor of $u$ different from $x$.
    Clearly, $\deg_B(v) = 2$ in this case, as otherwise $ux$ could be recolored blue and the resulting coloring would be a 2-liec.
    However, if $\deg_B(v) = 2$, we can recolor $ux$ blue, and double it, and the resulting coloring would be a 2-liec of $T + ux$.

    \textbf{Case 2.} Let $\deg_T(u) = 3$.
    Denote by $v_1$ and $v_2$ the neighbors of $u$ different from $x$.
    Similarly to Case 1, at least one of the neighbors of $u$, without loss of generality $v_1$, has the blue degree three, i.e. $\deg_B(v_1) = 3$, as otherwise we could recolor $ux$ blue and obtain a 2-liec of $T$.
    If $\deg_B(v_2) \neq 4$, recoloring of $ux$ blue and doubling it yields a 2-liec of $T+ux$.
    If $\deg_B(v_2) = 4$ then recolor $ux$ blue, but double the edge $uv_2$. 
    In this case, $\deg_B(x) = 1$, $\deg_B(u) = 4$, $\deg_B(v_1) = 3$, $\deg_B(v_1) = 3$, and the blue degree of every other vertex is at most four, since $\Delta(T) \leq 4$. Hence, the resulting coloring is a 2-liec of $T+uv_2$. 

    \textbf{Case 3.} Let $\deg_B(u) = 4$. In this case, recolor $ux$ blue and double it. The resulting coloring is a 2-liec of $T+ux$, since the blue degree of $u$ is five, and the blue degree of every other vertex is bounded by $\Delta(T)=4$.
\end{proof}

Theorem~\ref{paths} and Theorem~\ref{cycles} describe the value of $\mathcal{D}_\mathrm{lir}(G)$ in the case when $G$ is a path or a cycle. 
To fully cover all locally irregular uncolorable graphs, the graphs from the recursively defined class $\mathfrak{T}$ remain.
We show that for a graph  $G$ from $\mathfrak{T},$  there is a set of edges that need to be doubled so the resulting multigraph has a 2-liec in which parallel edges are colored the same, and
we also prove that the size of such a set is upper-bounded by the number of cycles of $G$.
However, the approach presented to prove this can be generalized for all cactus graphs defined in a similar way to those in $\mathfrak{T}$.

Let $\mathfrak{T}^*$ be the family of graphs defined recursively in the following way:
\begin{itemize}
  \item the cycle $C_n$ belongs to $\mathfrak{T}^*$,
  \item if $G$ is a graph from $\mathfrak{T}^*$, then any graph $G'$ obtained from $G$ by identifying a vertex $v\in V(G)$ of degree two, which belongs to a cycle in $G$, with an end vertex of a path of positive length or with an end vertex of a path such that the other end vertex of that path is identified with a vertex of a new cycle belongs to $\mathfrak{T}^*$.
\end{itemize}

As $\mathfrak{T}^*$ contains cycles, besides other cacti, and some cycles need two doublings, we exclude them from the following theorem (for $\mathcal{D}_\mathrm{lir}(C_n)$ see Theorem~\ref{cycles}).

\begin{tw}\label{special_cacti}
    Let $G \in \mathfrak{T}^*$ be different from a cycle.
    If $k$ is the number of cycles of $G$, then there is an independent set $E_d$ of at most $k$ non-pendant edges such that $G + E_d$ have a $2$-liec.
\end{tw}
\begin{proof}
    Suppose first that $G$ is unicyclic.
    We present a general way, how to color $G$. The properties of such a coloring will be later used in an inductive step.

    Let $C = v_1, \dots, v_n$ be a cycle of $G$.
    Since $G$ is different from a cycle, without loss of generality suppose that $v_1$ is of degree three in $G$.
    Denote by $P_i$ the pendant path attached to $v_i$ for all $i \in \{1, \dots, n\}$, and by $\ell_i$ the length of $P_i$.
    In the following, we suppose that $\ell_i \in \{0,1,2\}$ for each $i \in \{1, \dots, n\}$, as it is easy to see that a 2-liec of a multigraph $M$ with a pendant path of length $\ell \in \{1,2\}$ ending at a pendant vertex $v$ can be extended into a 2-liec of the multigraph obtained from $M$ by adding a vertex-disjoint path of even length to $M$ and identifying one of its end vertices with $v$ (simply alternate blue and red edges on pairs of adjacent edges of the added path), see Figure~\ref{fig:extension_pendant_paths} for illustration. 

    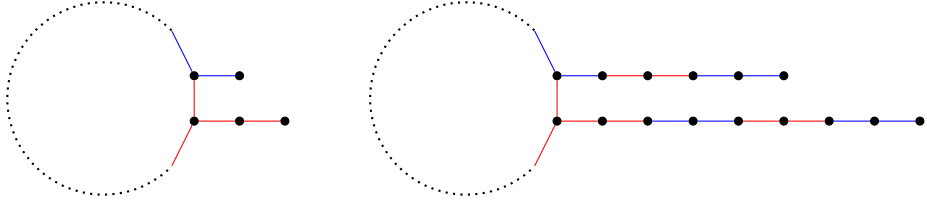
\begin{figure}[h]
        \centering
        \begin{tikzpicture}[scale=0.6]
        \begin{pgfonlayer}{nodelayer}
        \node [style={black_dot}] (0) at (0, 0.5) {};
        \node [style={black_dot}] (1) at (0, -0.5) {};
        \node [style=empty] (2) at (-0.5, 1.5) {};
        \node [style=empty] (3) at (-0.5, -1.5) {};
        \node [style={black_dot}] (4) at (1, -0.5) {};
        \node [style={black_dot}] (5) at (2, -0.5) {};
        \node [style={black_dot}] (6) at (1, 0.5) {};
        \node [style={black_dot}] (7) at (8, 0.5) {};
        \node [style={black_dot}] (8) at (8, -0.5) {};
        \node [style=empty] (9) at (7.5, 1.5) {};
        \node [style=empty] (10) at (7.5, -1.5) {};
        \node [style={black_dot}] (11) at (9, -0.5) {};
        \node [style={black_dot}] (12) at (10, -0.5) {};
        \node [style={black_dot}] (13) at (9, 0.5) {};
        \node [style={black_dot}] (14) at (11, -0.5) {};
        \node [style={black_dot}] (15) at (12, -0.5) {};
        \node [style={black_dot}] (16) at (13, -0.5) {};
        \node [style={black_dot}] (17) at (14, -0.5) {};
        \node [style={black_dot}] (18) at (15, -0.5) {};
        \node [style={black_dot}] (19) at (16, -0.5) {};
        \node [style={black_dot}] (20) at (10, 0.5) {};
        \node [style={black_dot}] (21) at (11, 0.5) {};
        \node [style={black_dot}] (22) at (12, 0.5) {};
        \node [style={black_dot}] (23) at (13, 0.5) {};
        \end{pgfonlayer}
        \begin{pgfonlayer}{edgelayer}
        \draw [style={red_edge}] (3) to (1);
        \draw [style={red_edge}] (1) to (0);
        \draw [style={red_edge}] (1) to (4);
        \draw [style={red_edge}] (4) to (5);
        \draw [style={blue_edge}] (0) to (2);
        \draw [style={blue_edge}] (0) to (6);
        \draw [style={red_edge}] (10) to (8);
        \draw [style={red_edge}] (8) to (7);
        \draw [style={red_edge}] (8) to (11);
        \draw [style={red_edge}] (11) to (12);
        \draw [style={blue_edge}] (7) to (9);
        \draw [style={blue_edge}] (7) to (13);
        \draw [style={red_edge}] (13) to (20);
        \draw [style={red_edge}] (20) to (21);
        \draw [style={red_edge}] (15) to (16);
        \draw [style={red_edge}] (16) to (17);
        \draw [style={blue_edge}] (21) to (22);
        \draw [style={blue_edge}] (22) to (23);
        \draw [style={blue_edge}] (12) to (14);
        \draw [style={blue_edge}] (14) to (15);
        \draw [style={blue_edge}] (17) to (18);
        \draw [style={blue_edge}] (18) to (19);
        \end{pgfonlayer}
        \draw[dotted, thick] (-0.5,1.5) arc (45:315:2.12132);
        \draw[dotted, thick] (7.5,1.5) arc (45:315:2.12132);
        \end{tikzpicture}
        \caption{Example of an extension of pendant paths considered in the proof of Theorem~\ref{special_cacti}.}
        \label{fig:extension_pendant_paths}
    \end{figure}

    Start by coloring $v_n v_1$, $v_1v_2$, and the edges of $P_1$ blue.
    Then continue by coloring $v_iv_{i+1}$ and the edges of $P_i$ for each $i \in \{2, \dots, n-1\}$ in the following way:
    \begin{itemize}
        \item If $v_{i-1} v_i$ is blue (red) and $\deg_B(v_{i-1}) = 1$ ($\deg_R(v_{i-1}) = 1$), color $v_iv_{i+1}$ and $P_i$ blue (red).
        \item If $v_{i-1} v_i$ is blue (red), $\deg_B(v_{i-1}) \in \{2,3\}$ ($\deg_R(v_{i-1}) \in \{2,3\}$), and $\ell_i \in \{0,1\}$, color $v_iv_{i+1}$ and $P_i$ red (blue).
        \item If $v_{i-1} v_i$ is blue (red), $\deg_B(v_{i-1}) = 2$ ($\deg_R(v_{i-1}) = 2$), and $\ell_i = 2$, color $v_i v_{i+1}$ and $P_i$ blue (red).
        \item If $v_{i-1} v_i$ is blue (red), $\deg_B(v_{i-1}) = 3$ ($\deg_R(v_{i-1}) = 3$), and $\ell_i = 2$, color $v_i v_{i+1}$ blue (red) and color $P_i$ red (blue).
    \end{itemize}
    After $v_{n-1}v_n$ was colored, the only uncolored part of $G$ is $P_n$ (if $\ell_n \neq 0$).
    Moreover, it is easy to see that the local irregularity condition holds for almost all edges, with a possible exception of edges $v_{n-1} v_n$ or $v_n v_1$.
    Hence, to transform the obtained partial coloring into a 2-liec, we will only consider the part of $G$ that is close to the edges $v_{n-1}v_n$ and  $v_n v_1$.
    Note that there are 18 possibilities to start with; these possibilities differ in the color of $v_{n-1}v_n$, color degree of $v_{n-1}$ (for the color that is used on $v_{n-1} v_n$), and $\ell_n$. 
    However, a lot of these cases can be solved rather easily by only coloring $P_n$ blue or red (for example if $v_{n-1}v_n$ is red, $\deg_R(v_{n-1}) \neq 1$, and $\ell_n = 0$, then nothing needs to be done, or if $v_{n-1}v_n$ is blue, $\deg_B(v_{n-1}) = 1$, and $\ell_n = 2$, then color $P_n$ red, etc.); we will leave these cases on the reader and we will hence discuss only the problematic cases when some edges need to be recolored or doubled.

    \noindent\textbf{Case 1.} Suppose that $v_{n-1}v_n$ is red, $\deg_R(v_{n-1}) = 1$, and $\ell_n=0$. In this case, recolor $v_n v_1$ to red. If, however, $\ell_1 = 2$ then double the edge $v_1 v_2$. Note that if $\ell_1 = 1$, no doubling was used and $\deg_B(v_1) = 2$, but one could double the edge $v_1v_2$ to create a 2-liec with $\deg_B(v_1) = 3$.

    \noindent\textbf{Case 2.} Suppose that $v_{n-1} v_n$ is blue, $\deg_B(v_{n-1}) = 2$, and $\ell_n = 0$. In this case, double the edge $v_n v_{n-1}$. Note that $\deg_B(v_1) = 4$. For the general step of the mathematical induction, however, we need to have also a different way, how to obtain a 2-liec in this case when $\ell_1 = 1$, where the blue degree of $v_1$ is three. Hence, we distinguish several cases; in all of them suppose that $\ell_1 = 1$:
    
    \noindent\textbf{Subcase 2.1.} Suppose that $v_{n-2} v_{n-1}$ is blue and $\deg_B(v_{n-2}) = 1$. In this case, $\ell_{n-1} = 0$ (otherwise, from the construction of the coloring we would have $\deg_B(n-1) = 3$), $v_{n-3}v_{n-2}$ is red and $\deg_R(v_{n-3}) \in \{2,3\}$. Recolor $v_{n-2} v_{n-1}$ by red, and double it if $\deg_R(v_{n-3}) = 2$. The resulting coloring is a 2-liec in which $\deg_B(v_1) = 3$. 


    \noindent\textbf{Subcase 2.2.}  Suppose that $v_{n-2} v_{n-1}$ is blue and $\deg_B(v_{n-2}) = 3$. In this case we have $\ell_{n-1} = 2$, and $P_{n-1}$ is red. Recolor $v_{n-1} v_n$ to red and double it. In the resulting coloring, $\deg_B(v_1) = 3$.

    \noindent\textbf{Subcase 2.3.} Suppose that $v_{n-2} v_{n-1}$ is red. Then $\deg_R(v_{n-2}) \in \{2,3\}$, $\ell_{n-1} = 1$ and $P_{n-1}$ is blue. in this case recolor $v_{n-1} v_n$ and $P_{n-1}$ to red, and double the edge $v_{n-1} v_n$ if $\deg_R(v_{n-2}) = 3$. In the resulting coloring, $\deg_B(v_1) = 3$.

    \noindent\textbf{Case 3.} Suppose that $v_{n-1} v_n$ is blue, $\ell_{n} = 1$, and $\deg_B(v_{n-1}) = 1$. In this case color $P_n$ blue and then double the edge $v_1v_2$ in order to create a 2-liec in which $\deg_B(v_1) = 4$. In order to create a 2-liec in which $\deg_B(v_1) = 3$, color $P_n$ blue and double the edge $v_{n-1} v_n$.

    \noindent\textbf{Case 4.} Suppose that $v_{n-1} v_n$ is blue, $\ell_{n} = 1$, and $\deg_B(v_{n-1}) \in \{2,3\}$. In this case, color $P_n$ red, and recolor $v_nv_1$ red. If $\ell_1 = 2$, double the edge $v_1 v_2$.
    If $\ell_1 =1$, $\deg_B(v_1) = 2$ and no edge is doubled (to construct a 2-liec with $\deg_B(v_1) = 3$ in this case, double the edge $v_1v_2$).

    \noindent\textbf{Case 5.} Suppose that $v_{n-1} v_n$ is blue, $\ell_{n} = 2$, and $\deg_B(v_{n-1}) = 2$. If $\ell_1 = 2$, color $P_n$ blue and double the edge $v_1v_2$; in this case $\deg_B(v_1) = 4$.
    However, if $\ell_1 = 1$, color $P_n$ red, recolor $v_n v_1$ to red, and double $v_nv_1$ in order to create a 2-liec in which $\deg_B(v_1) = 2$, or color $P_n$ blue and double $v_1v_2$ in order to create a 2-liec with $\deg_B(v_1) = 4$.

    \noindent\textbf{Case 6.} Suppose that $v_{n-1} v_n$ is red, $\ell_n = 2$, and $\deg_R(v_{n-1}) \in \{1,2\}$. In this case, color $P_n$ red and recolor $v_n v_1$  to red. If $\ell_1 = 2$, double the edge $v_1v_2$.
    However, if $\ell_1 = 1$, $\deg_B(v_1) = 2$ and no edge is doubled (to construct a 2-liec with $\deg_B(v_1) = 3$ in this case, double the edge $v_1v_2$).

    \noindent\textbf{Case 7.} Suppose that $v_{n-1} v_n$ is red, $\ell_n = 2$, and $\deg_R(v_{n-1}) = 3$. In this case, color $P_n$ blue and recolor $v_nv_1$ to red. If $\ell_1 = 2$, double the edge $v_1v_2$.
    However, if $\ell_1 = 1$, $\deg_B(v_1) = 2$ and no edge is doubled (to construct a 2-liec with $\deg_B(v_1) = 3$ in this case, double the edge $v_1v_2$).

    Cases 1 -- 7 describe all nontrivial cases of how to finish the produced coloring of $G$ in such a way that the resulting coloring is a 2-liec and no more than one edge is doubled.
    It is also easy to see that the double edge is not a pendant edge in any of the cases.
    Moreover, it is worth mentioning that if $\ell_1 = 1$, there is a 2-liec of $G$ with at most one edge doubled, in which $\deg_B(v_1) = 3$, and there is a 2-liec of $G$ with at most one edge doubled, in which $\deg_B(v_1) \neq 3$ (for the cases that are not described in Case 1 -- 7, in order to create a 2-liec in which $\deg_B(v_1) = 4$, double the edge $v_1 v_2$, as there was not any doubling before and $\deg_B(v_2)$ was 1). 

    Consider now the case when $k \geq 2$; assume that for each $G' \in \mathfrak{T}^*$ with $k' < k$ cycles, there is a an independent set of non-pendant edges $E_d'$ of size at most $k'$ such that $G' + E_d'$ has a 2-liec, or $G'$ is a cycle.

    Let $C$ be a pendant cycle of $G$, i.e., a cycle that is connected to only one other cycle of $G$ by a path $P$ with inner vertices of degree two in $G$ (the existence of such a cycle follows from the definition of graphs in $\mathfrak{T}^*$).
    Denote by $v_C$ and $v_R$ the end vertices of $P$ such that $v_C \in V(C)$ and $v_R$ is the vertex of the other cycle of $G$.
    Denote by $\ell$ the length of $P$.
    
    First, consider the case when $\ell \geq 2$. 
    Let $u$ be a neighbor of $v_C$ on $P$.
    Split $G$ into two connected subgraphs $G_1$ and $G_2$ such that $G = G_1 \cup G_2$, $\{u\} = V(G_1) \cap V(G_2)$, and $C \subseteq G_1$.
    Clearly, $G_1$ is unicyclic and $G_2$ has $k-1$ cycles, and both $G_1$ and $G_2$ are different from cycles, as they have pendant paths of non-zero lengths. 
    From the assumption, we have that there are sets $E_1$ and $E_2$ of at most one and at most $k-1$ independent non-pendant edges of $G_1$ and $G_2$, respectively, such that multigraphs $G_1 + E_1$ and $G_2 + E_2$ are locally irregular 2-colorable.
    Take a 2-liec of $G_1+E_1$ in which $uv_C$ is blue and take a 2-liec of $G_2+E_2$ in which the edge incident to $u$ is red, and combine these two colorings into a 2-liec of $G + (E_1 \cup E_2)$.
    Clearly, no edge in $E_1 \cup E_2$ is pendant and, since edges of $E_1$ and $E_2$ were not pendant in $G_1$ and $G_2$ (hence, none of the edges incident to $u$ is in $E_1$ or $E_2$), $E_1 \cup E_2$ is an independent set. 

    In the following, suppose that $\ell = 1$.
    Let $G_1$ and $G_2$ be connected subgraphs of $G$ such that $G = G_1 \cup G_2$, $V(G_1) \cap V(G_2) = \{v_C, v_R\}$, and $C \subseteq G_1$.
    From the induction assumption we have that there is an independent set of at most $k-1$ non-pendant edges $E_2$ of $G_2$ such that $G_2 + E_2$ has a 2-liec $\varphi$.
    Without loss of generality assume that the edge $v_C v_R$ is blue in $\varphi$; denote by $q$ the blue degree of $v_R$ in $\varphi$.
    Clearly $q \in \{2,3,4\}$.

    If $q = 2$, let $G_1'$ be a graph created from $G'$ by adding a new vertex $x$ and an edge $v_R x$.
    Let $\psi$ be a 2-liec of $G_1'$ obtained in the same way as described above, starting from $v_C$.
    Hence, $\deg_B(v_R) = 2$ in $\varphi$, and one can easily combine $\varphi$ and $\psi$ in order to create a 2-liec of $G$: for each edge $e \in E(G)$ let $\xi(e) = \varphi(e)$ if $e \in E(G_2)$ and $\xi(e) = \psi(e)$ if $e \in E(G_1)$.

    If $q = 3$, take a 2-liec of $G_1$ (with a doubled edge, if necessary) in which $v_Cv_R$ is blue and $\deg_B(v_C) \neq 3$, and if $q = 4$ take a 2-liec of $G_1$ in which $v_Cv_R$ is blue and $\deg_B(v_C) \neq 4$; we already showed the existence of such colorings. 
    Combine such a coloring of $G_1$ with the coloring of $G_2$.  
\end{proof}

For a graph $G$ from the family $\mathfrak{T}$, let $G_\triangle$ denote the tree whose vertices are the triangles of $G$, and two vertices are adjacent in $G_\triangle$ if and only if an odd-length path in $G$ connects the corresponding triangles. 
The triangles of $G$ that correspond to pendant vertices of $G_\triangle$ will be called pendant triangles of $G$.
We show that the number of pendant triangles of $G$ gives a lower bound on $\mathcal{D}_\mathrm{lir}(G)$ which implies that $\mathcal{D}_\mathrm{lir}$ is not upper-bounded by a constant in general. 

\begin{lem}\label{lemma_doublings_on_triangles}
    If $G \in \mathfrak{T}$ and it has at least three triangles then $\mathcal{D}_\mathrm{lir}(G)$ is at least the number of pendant triangles without vertices of degree two in $G$.
\end{lem}
\begin{proof}
    Let $H_0$ be a pendant triangle in $G$ whose all vertices are of degree three.
    Denote by $x_0, y_0, z_0$ the vertices of $H$ in such a way that $y_0$ and $z_0$ are the end vertices of pendant paths in $G$; let $y_0, y_1, \dots, y_{2p}$ and $z_0, z_1, \dots, z_{2q}$ be the vertices on these paths.
    Let $x_1$ be the neighbor of $x_0$ different from $y_0$ and $z_0$.

    Let $H_1$ be the subgraph of $G$ induced on the vertices $x_0,x_1, y_0, y_1 , \dots, y_{2p}, z_0, z_1, \dots, z_{2q}$, i.e., $H_1 = G[x_0,x_1, y_0, y_1 , \dots, y_{2p}, z_0, z_1, \dots, z_{2q}]$. 
    By $E_d$ denote any set of $\mathcal{D}_\mathrm{lir}(G)$ edges such that $G + E_d$ has a 2-liec in which parallel edges are colored the same; we will show that $E_d \cap E(H_1) \neq \emptyset$.
    In the following consider a 2-liec of $G + E_d$.

    Suppose to the contrary that $E_d \cap E(H_1) = \emptyset$.
    Since 2-liec of a path of even length is unique (up to interchange of colors), it is sufficient to consider the graph $H_2 = G[x_0,x_1, y_0, y_1, y_2, z_0, z_1,z_2]$ and distinguish two cases depending on whether the color used on $y_0y_1$ and $y_1y_2$ is the same as the color used on $z_0z_1$ and $z_1z_2$, or not.

    Suppose first that edges $y_0y_1, y_1y_2, z_0z_1, z_1z_2$ are red.
    Then either none or both of the edges $x_0y_0, y_0z_0$ are red, as otherwise  $\deg_R(y_0) = \deg_R(y_1)$.
    Similarly, either none or both of the edges $x_0z_0, y_0z_0$ are colored red.
    Combining these two observations we get that the triangle $H_0$ is monochromatic. 
    This however yields a contradiction since either $\deg_R(y_0) = 3 = \deg_R(z_0)$ (if $y_0z_0$ is red) or $\deg_B(y_0) = 2 = \deg_B(z_0)$ (if $y_0z_0$ is blue).

    Suppose next that $y_0y_1, y_1y_2$ are red and  $z_0z_1,z_1z_2$ are blue.
    If exactly one of the edges $x_0y_0, y_0z_0$ was red then $\deg_R(y_0) = \deg_R(y_1)$, 
    and similarly, if exactly one of the edges $x_0z_0, y_0z_0$ was colored blue then $\deg_B(z_0) = \deg_B(z_1)$.
    It follows from these observations that $H_0$ is monochromatic in this case also.
    If, on the one hand, $x_0y_0,x_0z_0,y_0z_0$ are red then either $\deg_R(x_0) = \deg_R(y_0) = 3$ (if $x_0x_1$ is red) or $\deg_R(x_0) = \deg_R(z_0) = 2$ (if $x_0x_1$ is blue).
    If, on the other hand, $x_0y_0,x_0z_0,y_0z_0$ are blue then either $\deg_B(x_0) = \deg_B(y_0) = 2$ (if $x_0x_1$ is red) or $\deg_B(x_0) = \deg_B(z_0) = 2$ (if $x_0x_1$ is blue).

    Hence, in either case, the considered coloring violates the locally irregular coloring condition, proving the fact that $E_d \cap E(H_1) \neq \emptyset$.
\end{proof}

As an immediate corollary of Lemma~\ref{lemma_doublings_on_triangles} we get:
\begin{cor}\label{corollary_D_not_constant}
    For every constant $k$ there is a graph $G$ such that $\mathcal{D}_{\mathrm{lir}}(G) \geq k$. 
\end{cor}

For the subgraph $H_2$ of $G$ in the proof of Lemma~\ref{lemma_doublings_on_triangles} we showed that at least one doubling on edges of $H_2$ is always needed.
This allows us to construct graphs (in most cases outside of $\mathfrak{T}$) which need relatively many doublings, namely $\frac{1}{8}|E(G)|$ doublings.
To construct such graphs, consider several disjoint copies of $H_2$ and identify all copies of the vertex $x_1$ (see Figure~\ref{fig:1/8doubled_edges} for illustration).
Due to same reasons as in the proof of Lemma~\ref{lemma_doublings_on_triangles}, there is at least one edge that need to be doubled in each copy of $H_2$.
On the other hand, by doubling edges incident to $x_1$, one can obtain a 2-liec of the resulting multigraph where parallel edges are colored with the same color, see Figure~\ref{fig:1/8doubled_edges}.

\begin{figure}[h]
    \centering
        \begin{tikzpicture}
        \begin{pgfonlayer}{nodelayer}
        \node [style={black_dot}] (0) at (0, 0) {};
        \node [style={black_dot}] (1) at (-1, 0) {};
        \node [style={black_dot}] (2) at (0, 1) {};
        \node [style={black_dot}] (3) at (1, 0) {};
        \node [style={black_dot}] (4) at (1.75, 0.5) {};
        \node [style={black_dot}] (5) at (1.75, -0.5) {};
        \node [style={black_dot}] (6) at (2.5, 0.5) {};
        \node [style={black_dot}] (7) at (2.5, -0.5) {};
        \node [style={black_dot}] (8) at (3.25, -0.5) {};
        \node [style={black_dot}] (9) at (3.25, 0.5) {};
        \node [style={black_dot}] (10) at (-1.75, 0.5) {};
        \node [style={black_dot}] (11) at (-1.75, -0.5) {};
        \node [style={black_dot}] (12) at (-2.5, -0.5) {};
        \node [style={black_dot}] (13) at (-2.5, 0.5) {};
        \node [style={black_dot}] (14) at (-3.25, 0.5) {};
        \node [style={black_dot}] (15) at (-3.25, -0.5) {};
        \node [style={black_dot}] (16) at (-0.5, 1.75) {};
        \node [style={black_dot}] (17) at (0.5, 1.75) {};
        \node [style={black_dot}] (18) at (-0.5, 2.5) {};
        \node [style={black_dot}] (19) at (-0.5, 3.25) {};
        \node [style={black_dot}] (20) at (0.5, 3.25) {};
        \node [style={black_dot}] (21) at (0.5, 2.5) {};
        \end{pgfonlayer}
        \begin{pgfonlayer}{edgelayer}
        \draw [style={blue_edge}] (14) to (13);
        \draw [style={blue_edge}] (13) to (10);
        \draw [style={blue_edge}] (5) to (7);
        \draw [style={blue_edge}] (7) to (8);
        \draw [style={blue_edge}] (17) to (21);
        \draw [style={blue_edge}] (21) to (20);
        \draw [style={red_edge}] (19) to (18);
        \draw [style={red_edge}] (18) to (16);
        \draw [style={red_edge}] (16) to (17);
        \draw [style={red_edge}] (17) to (2);
        \draw [style={red_edge}] (2) to (16);
        \draw [style={red_edge}] (10) to (11);
        \draw [style={red_edge}] (1) to (11);
        \draw [style={red_edge}] (1) to (10);
        \draw [style={red_edge}] (11) to (12);
        \draw [style={red_edge}] (12) to (15);
        \draw [style={red_edge}] (3) to (4);
        \draw [style={red_edge}] (3) to (5);
        \draw [style={red_edge}] (4) to (6);
        \draw [style={red_edge}] (4) to (5);
        \draw [style={red_edge}] (6) to (9);
        \draw [style={red_edge}, bend left=330] (3) to (0);
        \draw [style={red_edge}, bend left] (2) to (0);
        \draw [style={red_edge}, bend left] (0) to (1);
        \draw [style={red_edge}, bend left] (1) to (0);
        \draw [style={red_edge}, bend left] (0) to (2);
        \draw [style={red_edge}, bend right] (0) to (3);
        \end{pgfonlayer}
        \end{tikzpicture}
    \caption{Example of a graph $G$ which needs $\frac{1}{8}|E(G)|$ doubled edges.}
    \label{fig:1/8doubled_edges}
\end{figure}
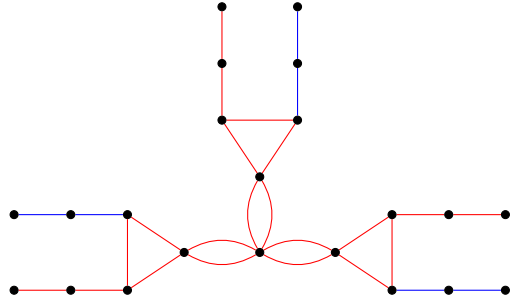

Note that in the constructions of graphs which need relatively many doublings, for example graphs considered in Lemma~\ref{lemma_doublings_on_triangles} and the construction illustrated in Figure~\ref{fig:1/8doubled_edges}, pendant paths of length two may be replaced by shrubs such that in every 2-liec of these shrubs, the root edge is a part of the monochromatic component in which the neighbor of the root has degree two (similarly to the neighbor of the end vertex of a path of even length in the case of graphs from $\mathfrak{T}$). Such an observation gives us more possibilities when constructing graphs $G$ for which $\mathcal{D}_\mathrm{lir}(G)$ exceeds any prescribed constant. However, replacing pendant paths of length two with such shrubs lower the ratio of the number of doubled edges to the number of all edges of a graph. Hence, the following question arises:
Is there a constant $k > \frac{1}{8}$ such that there are infinitely many connected graphs $G$ with $\mathcal{D}_\mathrm{lir}(G) \geq k |E(G)|$?

\section{Complete graphs, complete bipartite graphs and split graphs}

In~\cite{Grzelec wozniak}, an easy construction of the locally irregular 2-coloring of $^2G$, where $G$ is a complete $k$-partite graph ($k \geq 2$), was provided.
Even though the construction involves using red-blue edges, its use is limited to coloring the 3-partite subgraph of $G$. 
This allows building on this construction and, when considering another way of locally irregular 2-coloring of a complete 3-partite graph, a way to prove the following:
\begin{tw}
    If $G$ is complete $k$-partite graph different from $K_k$ then \linebreak ${\rm lir}(G) \leq 2$.
\end{tw}
\begin{proof}
    Let $G = K_{n_1, \dots, n_k}$ where $n_1 \geq \dots \geq n_k$, and let $A_1, \dots, A_k$ be the partition of vertices of $G$ into $k$ independent sets of sizes $n_1, \dots, n_k$. 
    The case when $k = 2$, i.e., $G$ is bipartite, was proved by Baudon, Bensmail, Przybyło, and Woźniak in~\cite{Baudon Bensmail Przybylo Wozniak}.
    
    Consider now that $G$ is a complete 3-partite graph.
    If $n_1 > n_2 > n_3$ then $G$ is a locally irregular graph, and thus, ${\rm lir}(G) = 1$; in the following suppose that all of its edges are colored red.
    
    If $n_1 = n_2 > n_3$ then color the edges between $A_1$ and $A_3$ by blue color and other edges by red color; it is easy to see that in this case, we have
    \begin{align*}
        \deg_B(v) = \left\{ \begin{matrix}
            n_3 &\text{if } v \in A_1, \\
            0 &\text{if } v \in A_2, \\
            n_1 &\text{if } v \in A_3, \\
        \end{matrix} \right.
        \quad
        \text{and}
        \quad
        \deg_R(v) = \left\{ \begin{matrix}
            n_1 &\text{if } v \in A_1, \\
            n_1+n_3 &\text{if } v \in A_2, \\
            n_1 &\text{if } v \in A_3. \\
        \end{matrix} \right.
    \end{align*}
    The only case when two adjacent vertices $u$ and $v$ in $G$ have the same red or blue degrees is the case when $u \in A_1$ and $v \in A_3$; in this case, however, the edge $uv$ is blue and $\deg_B(u) \neq \deg_B(v)$. 
    Hence, the coloring is locally irregular 2-coloring of $G$.

    Next, consider the case when $n_1 > n_2 = n_3$.
    Color the edges between $A_1$ and $A_2$ blue and other edges red.
    Similar to the previous case, we have
    \begin{align*}
        \deg_B(v) = \left\{ \begin{matrix}
            n_2 &\text{if } v \in A_1, \\
            n_1 &\text{if } v \in A_2, \\
            0 &\text{if } v \in A_3, \\
        \end{matrix} \right.
        \quad
        \text{and}
        \quad
        \deg_R(v) = \left\{ \begin{matrix}
            n_2 &\text{if } v \in A_1, \\
            n_2 &\text{if } v \in A_2, \\
            n_1+n_2 &\text{if } v \in A_3. \\
        \end{matrix} \right.        
    \end{align*}
    Even though $\deg_R(u) = \deg_R(v)$ for each $u \in A_1$ and $v \in A_2$, from the construction we have that $uv$ is blue. 
    Hence, the coloring is a locally irregular coloring of $G$.

    Consider now the case when $n_1 = n_2 = n_3$ (note that $n_1 \geq 2$ as $G$ is not a complete graph). 
    Let $M$ be a perfect matching of a subgraph of $G$ induced on $A_1 \cup A_3$.
    Color every edge between $A_1$ and $A_2$ and every edge of $M$ blue, and color the remaining edges red.
    We have
    \begin{align*}
        \deg_B(v) = \left\{ \begin{matrix}
            n_1 + 1 &\text{if } v \in A_1, \\
            n_1 &\text{if } v \in A_2, \\
            1 &\text{if } v \in A_3, \\
        \end{matrix} \right.
        \quad
        \text{and}
        \quad
        \deg_R(v) = \left\{ \begin{matrix}
            n_1 - 1 &\text{if } v \in A_1, \\
            n_1 &\text{if } v \in A_2, \\
            2n_1-1 &\text{if } v \in A_3. \\
        \end{matrix} \right.
    \end{align*}
    The coloring is a locally irregular 2-coloring of $G$ (since $n_1 \geq 2$).

    So far we proved that if $k \leq 3$ then $G$ has a 2-liec.
    It is also noteworthy that in each presented coloring of a complete 3-partite graph, each vertex has a non-zero red degree, hence, the blue degree of each vertex is upper bounded by $n_1+n_2 - 1$. 
    For $k \geq 4$ we use the inductive construction based on the construction for 2-multigraphs in~\cite{Grzelec wozniak}:
    \begin{itemize}
        \item if $k$ is even, find a locally irregular coloring of $K_{n_1, \dots, n_{k-1}}$ in which the red degree of every vertex is non-zero, and color every edge incident to $A_k$ blue;
        \item if $k$ is odd, find a locally irregular coloring of $K_{n_1, \dots, n_{k-1}}$ in which the blue degree of every vertex is non-zero, and color every edge incident to $A_k$ red.
    \end{itemize}
    Note that the resulting coloring of $K_{n_1, \dots, n_k}$ preserves the property that if $k$ is even then the blue degree of each vertex is non-zero, and if $k$ is odd then the red degree of each vertex is non-zero.
    Moreover, blue and red degrees of vertices of $K_{n_1, \dots, n_{k-1}}$ are raised by the same constants when the coloring of $K_{n_1, \dots, n_k}$ is created, hence, if the coloring of $K_{n_1, \dots, n_k}$ is not locally irregular, the conflict is on the edge incident to $A_k$.
    However, if $k$ is even then every edge incident to $A_k$ is colored blue, hence $\deg_B(v) = \sum_{i=1}^{k-1}n_i$ for each $v \in A_k$, and, on the other hand, $\deg_B(v) \leq n_k + \sum_{i=1}^{k-2}n_i$ for each $v \in \bigcup_{i=1}^{k-1}A_i$ (since $\deg_R(v) \geq 1$ for each such vertex). 
    Using the fact that $n_k \leq n_{k-1}$, it is easy to see that the coloring of $K_{n_1, \dots, n_k}$ is locally irregular.
    Essentially the same can be shown in the case when $k$ is odd; colors red and blue switch their roles in the argument.
\end{proof}

The previous theorem excluded complete graphs, which was expected since $\mathrm{lir}(K_n) = 3$ for each $n \geq 4$, see~\cite{Baudon Bensmail Przybylo Wozniak}.
For complete graphs, we therefore need to use at least one doubling to obtain a multigraph that admits a 2-liec.
We show that two doubled edges are sufficient, but except for the finite number of cases doubling of one edge suffices.

\begin{tw}\label{thm_Kn}
    For complete graph $K_n$ of order $n \geq 4$ we have
    \begin{align*}
        \mathcal{D}_{\mathrm{lir}}(K_n) = \left\{ \begin{matrix} 2 &\text{if } 6 \leq n \leq 10, \\ 1 &\text{otherwise.}  \end{matrix} \right.
    \end{align*}
\end{tw}
\begin{proof}
    The coloring of $K_n$ for $n \in \{4,6,8,9\}$ with one or two edges doubled is shown in Figure~\ref{fig:complete_graphs}.
    Colorings of $K_5$, $K_7$ and $K_{10}$ with one or two edges doubled may be obtained from the colorings of $K_4$, $K_6$ and $K_9$ shown in Figure~\ref{fig:complete_graphs} by adding a new vertex and coloring every edge incident to it blue.
    Using a computer program to check every possible decomposition of $K_n$ ($n \leq 10$) into two graphs, one of which is locally irregular and the other one can be made locally irregular by doubling at most one edge, we showed that it is not enough to double one edge whenever $6 \leq n \leq 10$; hence, two doubled edges are necessary in those cases.

    Let the vertices of $K_n$ be denoted by $v_1, \dots, v_n$.
    Let $e = v_1v_2$.
    In the following, we will consider the multigraph $K_n + e$.

    Consider now the coloring of $K_{11}+e$ shown in Figure~\ref{fig:K_11}.
    If $n \geq 12$, the coloring of $K_n+e$ is obtained from the coloring of $K_{n-1}+e$ by 
    coloring every edge incident to $v_n$ red if $n \equiv 0 \pmod{2}$, or blue if $n \equiv 1 \pmod{2}$.

    If $n$ is even then for every $v \in V(K_n+e) \setminus \{v_n\}$ we have $\deg_R(v) \geq 1$ and $\deg_B(v) \leq n-2$ (since every edge incident to $v_n$ is colored red, and there is not any blue doubled edge).
    If $n$ is odd then, similarly to the previous case, for every $v \in V(K_n+e) \setminus \{v_n\}$ we have $\deg_B(v) \geq 1$ and $\deg_R \leq n-2$ (if $n \geq 13$ then all edges incident to $v_n$ are blue, and end vertices of the double red edge are incident to at least five blue edges, for $n = 11$ see Figure~\ref{fig:K_11}).

    Hence, if there are some vertices $v_i$ and $v_j$ such that $\deg_H(v_i) = \deg_H(v_j)$ for $H \in \{R,B\}$, then $i,j \leq n-1$ and, from the construction of the coloring, it is easy to see that $\deg_H(v_i) = \deg_H(v_j)$ in $K_{n-1}+e$ as well.
    Since, however, the coloring of $K_{11}+e$ in Figure~\ref{fig:K_11} is locally irregular, the constructed colorings of $K_n+e$ for $n \geq 12$ are locally irregular too.

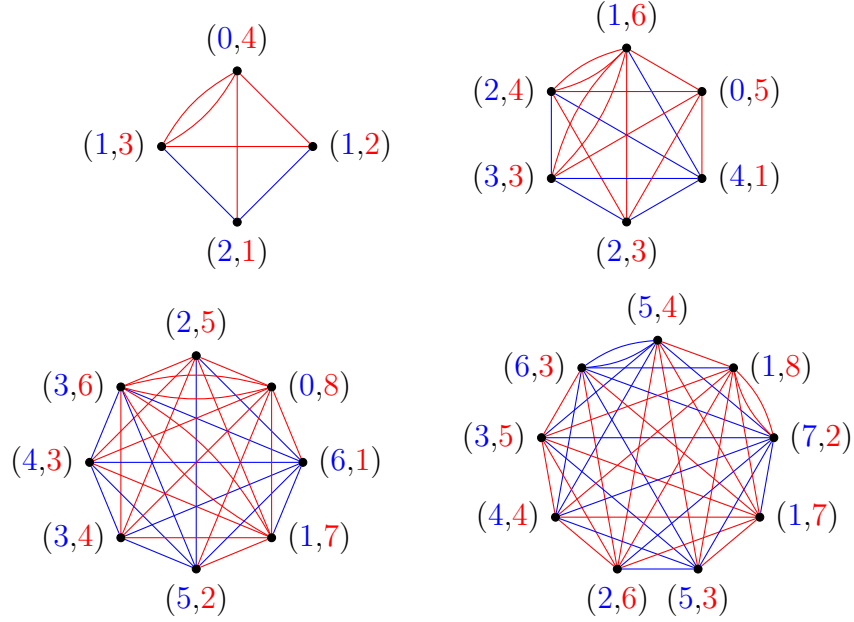
\begin{figure}[!ht]
    \centering

    \begin{tikzpicture}
    \begin{pgfonlayer}{nodelayer}
    \node [style={black_dot}, label={above:$ \color{black}{(}\color{blue}{0}\color{black}{,}\color{red}{4}\color{black}{)}$}] (0) at (90.00:1.00) {};
    \node [style={black_dot}, label={left:$ \color{black}{(}\color{blue}{1}\color{black}{,}\color{red}{3}\color{black}{)}$}] (1) at (180.00:1.00) {};
    \node [style={black_dot}, label={below:$ \color{black}{(}\color{blue}{2}\color{black}{,}\color{red}{1}\color{black}{)}$}] (2) at (270.00:1.00) {};
    \node [style={black_dot}, label={right:$ \color{black}{(}\color{blue}{1}\color{black}{,}\color{red}{2}\color{black}{)}$}] (3) at (360.00:1.00) {};
    \end{pgfonlayer}
    \begin{pgfonlayer}{edgelayer}
    \draw [style={red_edge}, bend right=15] (0) to (1);
    \draw [style={red_edge}, bend left=15] (0) to (1);
    \draw [style={red_edge}] (0) to (2);
    \draw [style={red_edge}] (0) to (3);
    \draw [style={red_edge}] (3) to (1);
    \draw [style={blue_edge}] (2) to (3);
    \draw [style={blue_edge}] (2) to (1);
    \end{pgfonlayer}
    \end{tikzpicture}
    \hspace{0.5cm}
    \begin{tikzpicture}
    \begin{pgfonlayer}{nodelayer}
    \node [style={black_dot}, label={above:$ \color{black}{(}\color{blue}{1}\color{black}{,}\color{red}{6}\color{black}{)}$}] (0) at (90.00:1.15) {};
    \node [style={black_dot}, label={left:$ \color{black}{(}\color{blue}{2}\color{black}{,}\color{red}{4}\color{black}{)}$}] (1) at (150.00:1.15) {};
    \node [style={black_dot}, label={left:$ \color{black}{(}\color{blue}{3}\color{black}{,}\color{red}{3}\color{black}{)}$}] (2) at (210.00:1.15) {};
    \node [style={black_dot}, label={below:$ \color{black}{(}\color{blue}{2}\color{black}{,}\color{red}{3}\color{black}{)}$}] (3) at (270.00:1.15) {};
    \node [style={black_dot}, label={right:$ \color{black}{(}\color{blue}{4}\color{black}{,}\color{red}{1}\color{black}{)}$}] (4) at (330.00:1.15) {};
    \node [style={black_dot}, label={right:$ \color{black}{(}\color{blue}{0}\color{black}{,}\color{red}{5}\color{black}{)}$}] (5) at (390.00:1.15) {};
    \end{pgfonlayer}
    \begin{pgfonlayer}{edgelayer}
    \draw [style={red_edge}, bend right=15] (0) to (1);
    \draw [style={red_edge}, bend left=15] (0) to (1);
    \draw [style={red_edge}, bend left=15] (0) to (2);
    \draw [style={red_edge}] (0) to (3);
    \draw [style={red_edge}] (3) to (1);
    \draw [style={blue_edge}] (2) to (3);
    \draw [style={blue_edge}] (2) to (1);
    \draw [style={blue_edge}] (2) to (4);
    \draw [style={blue_edge}] (4) to (3);
    \draw [style={blue_edge}] (4) to (1);
    \draw [style={blue_edge}] (4) to (0);
    \draw [style={red_edge}, bend left=15] (2) to (0);
    \draw [style={red_edge}] (5) to (2);
    \draw [style={red_edge}] (5) to (0);
    \draw [style={red_edge}] (5) to (1);
    \draw [style={red_edge}] (5) to (3);
    \draw [style={red_edge}] (5) to (4);
    \end{pgfonlayer}
    \end{tikzpicture}

    \begin{tikzpicture}
    \begin{pgfonlayer}{nodelayer}
    \node [style={black_dot}, label={above:$ \color{black}{(}\color{blue}{2}\color{black}{,}\color{red}{5}\color{black}{)}$}] (0) at (90.00:1.41) {};
    \node [style={black_dot}, label={left:$ \color{black}{(}\color{blue}{3}\color{black}{,}\color{red}{6}\color{black}{)}$}] (1) at (135.00:1.41) {};
    \node [style={black_dot}, label={left:$ \color{black}{(}\color{blue}{4}\color{black}{,}\color{red}{3}\color{black}{)}$}] (2) at (180.00:1.41) {};
    \node [style={black_dot}, label={left:$ \color{black}{(}\color{blue}{3}\color{black}{,}\color{red}{4}\color{black}{)}$}] (3) at (225.00:1.41) {};
    \node [style={black_dot}, label={below:$ \color{black}{(}\color{blue}{5}\color{black}{,}\color{red}{2}\color{black}{)}$}] (4) at (270.00:1.41) {};
    \node [style={black_dot}, label={right:$ \color{black}{(}\color{blue}{1}\color{black}{,}\color{red}{7}\color{black}{)}$}] (5) at (315.00:1.41) {};
    \node [style={black_dot}, label={right:$ \color{black}{(}\color{blue}{6}\color{black}{,}\color{red}{1}\color{black}{)}$}] (6) at (360.00:1.41) {};
    \node [style={black_dot}, label={right:$ \color{black}{(}\color{blue}{0}\color{black}{,}\color{red}{8}\color{black}{)}$}] (7) at (405.00:1.41) {};
    \end{pgfonlayer}
    \begin{pgfonlayer}{edgelayer}
    \draw [style={red_edge}] (0) to (1);
    \draw [style={red_edge}] (0) to (2);
    \draw [style={red_edge}] (0) to (3);
    \draw [style={red_edge}] (3) to (1);
    \draw [style={blue_edge}] (2) to (3);
    \draw [style={blue_edge}] (2) to (1);
    \draw [style={blue_edge}] (2) to (4);
    \draw [style={blue_edge}] (4) to (3);
    \draw [style={blue_edge}] (4) to (1);
    \draw [style={blue_edge}] (4) to (0);
    \draw [style={red_edge}] (5) to (2);
    \draw [style={red_edge}] (5) to (0);
    \draw [style={red_edge}, bend right=15] (5) to (1);
    \draw [style={red_edge}, bend left=15] (5) to (1);
    \draw [style={red_edge}] (5) to (3);
    \draw [style={red_edge}] (5) to (4);
    \draw [style={blue_edge}] (6) to (4);
    \draw [style={blue_edge}] (2) to (6);
    \draw [style={blue_edge}] (6) to (3);
    \draw [style={blue_edge}] (6) to (1);
    \draw [style={blue_edge}] (6) to (0);
    \draw [style={blue_edge}] (6) to (5);
    \draw [style={red_edge}] (7) to (6);
    \draw [style={red_edge}] (7) to (5);
    \draw [style={red_edge}] (7) to (4);
    \draw [style={red_edge}] (2) to (7);
    \draw [style={red_edge}] (7) to (3);
    \draw [style={red_edge}] (7) to (0);
    \draw [style={red_edge}, bend left=15] (7) to (1);
    \draw [style={red_edge}, bend right=15] (7) to (1);
    \end{pgfonlayer}
    \end{tikzpicture}
    \hspace{0.5cm}
    \begin{tikzpicture}
    \begin{pgfonlayer}{nodelayer}
    \node [style={black_dot}, label={above:$ \color{black}{(}\color{blue}{5}\color{black}{,}\color{red}{4}\color{black}{)}$}] (0) at (90.00:1.56) {};
    \node [style={black_dot}, label={left:$ \color{black}{(}\color{blue}{6}\color{black}{,}\color{red}{3}\color{black}{)}$}] (1) at (130.00:1.56) {};
    \node [style={black_dot}, label={left:$ \color{black}{(}\color{blue}{3}\color{black}{,}\color{red}{5}\color{black}{)}$}] (2) at (170.00:1.56) {};
    \node [style={black_dot}, label={left:$ \color{black}{(}\color{blue}{4}\color{black}{,}\color{red}{4}\color{black}{)}$}] (3) at (210.00:1.56) {};
    \node [style={black_dot}, label={below:$ \color{black}{(}\color{blue}{2}\color{black}{,}\color{red}{6}\color{black}{)}$}] (4) at (250.00:1.56) {};
    \node [style={black_dot}, label={below:$ \color{black}{(}\color{blue}{5}\color{black}{,}\color{red}{3}\color{black}{)}$}] (5) at (290.00:1.56) {};
    \node [style={black_dot}, label={right:$ \color{black}{(}\color{blue}{1}\color{black}{,}\color{red}{7}\color{black}{)}$}] (6) at (330.00:1.56) {};
    \node [style={black_dot}, label={right:$ \color{black}{(}\color{blue}{7}\color{black}{,}\color{red}{2}\color{black}{)}$}] (7) at (370.00:1.56) {};
    \node [style={black_dot}, label={right:$ \color{black}{(}\color{blue}{1}\color{black}{,}\color{red}{8}\color{black}{)}$}] (8) at (410.00:1.56) {};
    \end{pgfonlayer}
    \begin{pgfonlayer}{edgelayer}
    \draw [style={blue_edge}, bend right=15] (0) to (1);
    \draw [style={blue_edge}, bend left=15] (0) to (1);
    \draw [style={blue_edge}] (0) to (2);
    \draw [style={blue_edge}] (0) to (3);
    \draw [style={blue_edge}] (3) to (1);
    \draw [style={red_edge}] (2) to (3);
    \draw [style={red_edge}] (2) to (1);
    \draw [style={red_edge}] (2) to (4);
    \draw [style={red_edge}] (4) to (3);
    \draw [style={red_edge}] (4) to (1);
    \draw [style={red_edge}] (4) to (0);
    \draw [style={blue_edge}] (5) to (2);
    \draw [style={red_edge}] (5) to (0);
    \draw [style={blue_edge}] (5) to (1);
    \draw [style={blue_edge}] (5) to (3);
    \draw [style={blue_edge}] (5) to (4);
    \draw [style={red_edge}] (6) to (4);
    \draw [style={red_edge}] (2) to (6);
    \draw [style={red_edge}] (6) to (3);
    \draw [style={red_edge}] (6) to (1);
    \draw [style={red_edge}] (6) to (0);
    \draw [style={red_edge}] (6) to (5);
    \draw [style={blue_edge}] (7) to (6);
    \draw [style={blue_edge}] (7) to (5);
    \draw [style={blue_edge}] (7) to (4);
    \draw [style={blue_edge}] (2) to (7);
    \draw [style={blue_edge}] (7) to (3);
    \draw [style={blue_edge}] (7) to (0);
    \draw [style={blue_edge}] (7) to (1);
    \draw [style={blue_edge}] (8) to (1);
    \draw [style={red_edge}, bend left=15] (7) to (8);
    \draw [style={red_edge}, bend right=15] (7) to (8);
    \draw [style={red_edge}] (6) to (8);
    \draw [style={red_edge}] (8) to (5);
    \draw [style={red_edge}] (8) to (4);
    \draw [style={red_edge}] (8) to (3);
    \draw [style={red_edge}] (2) to (8);
    \draw [style={red_edge}] (8) to (0);
    \end{pgfonlayer}
    \end{tikzpicture}

    \caption{Locally irregular colorings of $K_n$ with at most two edges doubled, $n \in \{4,6,8,9\}$.}
    \label{fig:complete_graphs}
\end{figure}

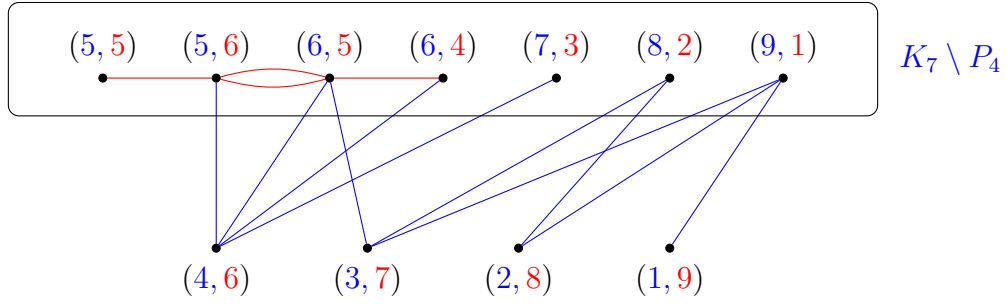
\begin{figure}[ht!]
    \centering
\begin{tikzpicture}
    \begin{pgfonlayer}{nodelayer}
    \node [style={black_dot}, label={above:($\color{blue}{5},\color{red}{5}$)}] (0) at (1.75, 0) {};
    \node [style={black_dot}, label={above:($\color{blue}{5},\color{red}{6}$)}] (1) at (3.25, 0) {};
    \node [style={black_dot}, label={above:($\color{blue}{6},\color{red}{5}$)}] (2) at (4.75, 0) {};
    \node [style={black_dot}, label={above:($\color{blue}{6},\color{red}{4}$)}] (3) at (6.25, 0) {};
    \node [style={black_dot}, label={above:($\color{blue}{7},\color{red}{3}$)}] (4) at (7.75, 0) {};
    \node [style={black_dot}, label={above:($\color{blue}{8},\color{red}{2}$)}] (5) at (9.25, 0) {};
    \node [style={black_dot}, label={above:($\color{blue}{9},\color{red}{1}$)}] (6) at (10.75, 0) {};
    \node [style={black_dot}, label={below:($\color{blue}{4},\color{red}{6}$)}] (9) at (3.25, -2.25) {};
    \node [style={black_dot}, label={below:($\color{blue}{3},\color{red}{7}$)}] (10) at (5.25, -2.25) {};
    \node [style={black_dot}, label={below:($\color{blue}{1},\color{red}{9}$)}] (11) at (9.25, -2.25) {};
    \node [style={black_dot}, label={below:($\color{blue}{2},\color{red}{8}$)}] (12) at (7.25, -2.25) {};
    \end{pgfonlayer}
    \begin{pgfonlayer}{edgelayer}
    \draw [style={red_edge}, bend left=15] (1) to (2);
    \draw [style={red_edge}, bend right=15] (1) to (2);
    \draw [style={red_edge}] (0) to (1);
    \draw [style={red_edge}] (2) to (3);
    \draw [style={blue_edge}] (11) to (6);
    \draw [style={blue_edge}] (5) to (12);
    \draw [style={blue_edge}] (12) to (6);
    \draw [style={blue_edge}] (10) to (6);
    \draw [style={blue_edge}] (10) to (5);
    \draw [style={blue_edge}] (10) to (2);
    \draw [style={blue_edge}] (9) to (4);
    \draw [style={blue_edge}] (9) to (3);
    \draw [style={blue_edge}] (9) to (2);
    \draw [style={blue_edge}] (9) to (1);
    \end{pgfonlayer}

    \draw[rounded corners] (0.5, -0.5) rectangle (12, 1) {};
    \node at (12,0.25) [label={right:$\color{blue} K_7 \setminus P_4$}] {};
\end{tikzpicture}
    \caption{Locally irregular coloring of $M_{11}$: all edges that are not displayed in $\color{blue} K_7 \setminus P_4$ are simple blue, and other omitted edges are simple red.}
    \label{fig:K_11}
\end{figure}

\end{proof}

A graph $G$ whose vertices can be partitioned into two disjoint sets $X$ and $Y$ such that $X$ is a clique and $Y$ is an independent set, is called \textit{split graph}.
Split graphs are, by definition, close to complete graphs. 
Hence, it is not a surprise that $\mathcal{D}_{\mathrm{lir}}(G)$ of any split graph $G$ is not greater than one, as it is the case for complete graphs (with finitely many exceptions).
To prove that $\mathcal{D}_\mathrm{lir}(G) \leq 1$ for every split graph $G$, we heavily use the results of Lintzmayer, Mota, and Sambinelli~\cite{Lintzmayer Mota Sambinelli}.
In~\cite{Lintzmayer Mota Sambinelli} it is precisely determined which split graphs admit 2-liec and which do not (see Theorem~1.5. and Theorem~3.1. in~\cite{Lintzmayer Mota Sambinelli}).
As a direct corollary of their result, we get the following lemma:
\begin{lem}\label{lemma_split}
    Let $G$ be a split graph different from a complete graph, with a maximum clique $X=\{v_1, \dots, v_n\}$ and an independent set $Y = V(G) \setminus X$ such that $\deg(v_1) \geq \dots \geq \deg(v_n)$.
    Let $d_i = |N(v_i) \cap Y|$ for each $i \in \{1, \dots, n\}$.
    Then $\mathcal{D}_{\mathrm{lir}}(G) \geq 1$ if one of the following is true:
    \begin{itemize}
        \item[$1)$] $G$ is a path of length three  
        \item[$2)$] $n \in \{4, \dots, 10\}$, $d_1 = 1$, $d_2=0$,
        \item[$3)$] $n \in \{6, \dots, 10\}$, $d_1 = 2$, $d_2=0$,
        \item[$4)$] $n \in \{8,9,10\}$, $d_1=3$, $d_2=0$, 
        \item[$5)$] $n \in \{6, 7, 8\}$, $d_1 = d_2= 1$, $d_3=0$,
        \item[$6)$] $n=10$, $d_1=4$, $d_2=0$,
        \item[$7)$] $n \geq 11$, $d_1 < \left\lfloor\tfrac{n}{2}\right\rfloor$, $d_2=0$.
    \end{itemize}
\end{lem}

The previous lemma shows that to fully solve the problem of determining  $\mathcal{D}_\mathrm{lir}$ for split graphs, it is sufficient to consider only a few small split graphs and one, very specific, infinite family.

\begin{tw}\label{theorem_split}
    If $G$ is a split graph different from a complete graph, then $\mathcal{D}_{\mathrm{lir}}(G) \leq 1$.
\end{tw}
\begin{proof}
    From Theorem~\ref{paths} we have that $\mathcal{D}_\mathrm{lir}(P_4) = 1$.
    We deal with the other split graphs listed in Lemma~\ref{lemma_split} in the following.

    Let $G$ be a split graph with $\mathrm{lir}(G) > 2$.
    Denote by $v_1, \dots, v_n$ the vertices of the largest clique in $G$.
    We distinguish several cases, covering every possibility listed in Lemma~\ref{lemma_split}.
    We will heavily use the colorings obtained in the proof of Theorem~\ref{thm_Kn} (see Figure~\ref{fig:complete_graphs}).

    Suppose first that $\mathcal{D}_{\mathrm{lir}}(K_n) = 1$. 
    Then take any locally irregular coloring $\varphi$ in which the red degree of $v_1$ is the maximum possible.
    Attach $d_1$ pendant red edges to $v_1$. 
    Such a coloring of an obtained split graph is locally irregular, which completes the proof that $\mathcal{D}_\mathrm{lir}(G) \leq 1$ whenever $n \in \{4,5\}$ or $n \geq 11$ (see Lemma~\ref{lemma_split} and Theorem~\ref{thm_Kn}).
    
    Now, suppose that $n \in \{6,7,8\}$, $d_1 = d_2 = 1$ and $d_3 = 0$.
    Consider the coloring of $K_n$ obtained in the proof of Theorem~\ref{thm_Kn}, where the edge $v_1v_2$ is one of the doubled red edges.
    Replace two parallel edges $v_1v_2$ with a single red edge, and add two pendant red edges incident to $v_1$ and $v_2$.
    Clearly, the coloring of the resulting multigraph is a 2-liec.
    For the second possibility in this case, when there is a common neighbor of $v_1$ and $v_2$ outside of the vertices of the clique in $G$, simply subdivide one of the parallel edges $v_1v_2$ in the coloring obtained in the proof of Theorem~\ref{thm_Kn}.
    This completes the case.

    When $n\in \{6,7\}$, $d_1 \in \{1,2\}$ and $d_2 = 0$, consider the coloring of the multigraph $K_n + v_1v_2 + v_2v_3$ from the proof of Theorem~\ref{thm_Kn}, in which $\deg_B(v_1) = 1$ and $\deg_B(v_2) = 3$ if $n=6$, or $\deg_B(v_1) = 2$ and $\deg_B(v_2) = 4$ if $n=7$ (i.e. $v_1v_2$ is one of the doubled red edges).
    Replace two red parallel edges $v_1v_2$ with a single red edge and add $d_1$ pendant red edges incident to $v_1$. The resulting coloring is a 2-liec of $G+ v_1v_2 + v_2v_3$.

    If $n \in \{9,10\}$, consider the coloring of $K_n + v_1v_2 + v_3v_4$ defined in the proof of Theorem~\ref{thm_Kn} in which the parallel edges $v_1v_2$ are colored red and $\deg_R(v_1) = 8$.
    If the parallel edges $v_1v_2$ are replaced with a single red edge and $d_1$ red pendant edges incident to $v_1$ are added, we obtain a 2-liec of $G+v_1v_2 + v_3v_4$.
    This completes the cases when $n \in \{9,10\}$, $d_1 \in \{1, 2,3,4\}$ and $d_2 = 0$.

    The left-over cases, namely when $n = 8$, $d_1 \in \{ 1,2,3\}$ and $d_2 = 0$ are dealt with independently to the coloring of $K_8 + e_1 + e_2$ from the proof of Theorem~\ref{thm_Kn}. For 2-liecs of $G$ with one doubling in these cases see Figure~\ref{fig:split8}. This completes the proof of the theorem.
    
    \begin{figure}[h]
        \centering
        \begin{tikzpicture}[scale=0.8]
        \begin{pgfonlayer}{nodelayer}
        \node [style={black_dot}, label={[xshift=-19pt, yshift=-8pt]$\color{black}{(}\color{blue}{2}\color{black}{,}\color{red}{9}\color{black}{)}$}] (0) at (90.00:1.41) {};
        \node [style={black_dot}, label={left:$\color{black}{(}\color{blue}{4}\color{black}{,}\color{red}{3}\color{black}{)}$}] (1) at (135.00:1.41) {};
        \node [style={black_dot}, label={left:$\color{black}{(}\color{blue}{3}\color{black}{,}\color{red}{4}\color{black}{)}$}] (2) at (180.00:1.41) {};
        \node [style={black_dot}, label={below:$\color{black}{(}\color{blue}{3}\color{black}{,}\color{red}{5}\color{black}{)}$}] (3) at (225.00:1.41) {};
        \node [style={black_dot}, label={below:$\color{black}{(}\color{blue}{5}\color{black}{,}\color{red}{2}\color{black}{)}$}] (4) at (270.00:1.41) {};
        \node [style={black_dot}, label={below:$\color{black}{(}\color{blue}{1}\color{black}{,}\color{red}{6}\color{black}{)}$}] (5) at (315.00:1.41) {};
        \node [style={black_dot}, label={right:$\color{black}{(}\color{blue}{6}\color{black}{,}\color{red}{1}\color{black}{)}$}] (6) at (360.00:1.41) {};
        \node [style={black_dot}, label={right:$\color{black}{(}\color{blue}{0}\color{black}{,}\color{red}{7}\color{black}{)}$}] (7) at (405.00:1.41) {};
        \node [style={black_dot}, label={right:$\color{black}{(}\color{blue}{0}\color{black}{,}\color{red}{1}\color{black}{)}$}] (8) at (75.00:2.41) {};
        \node [style={black_dot}, label={above:$\color{black}{(}\color{blue}{0}\color{black}{,}\color{red}{1}\color{black}{)}$}] (9) at (90.00:2.41) {};
        \node [style={black_dot}, label={left:$\color{black}{(}\color{blue}{0}\color{black}{,}\color{red}{1}\color{black}{)}$}] (10) at (105.00:2.41) {};
        \end{pgfonlayer}
        \begin{pgfonlayer}{edgelayer}
        \draw [style={red_edge}] (1) to (0);
        \draw [style={red_edge}, bend right=15] (0) to (3);
        \draw [style={red_edge}] (0) to (2);
        \draw [style={red_edge}, bend right=15] (3) to (0);
        \draw [style={red_edge}] (2) to (3);
        \draw [style={blue_edge}] (3) to (1);
        \draw [style={blue_edge}] (1) to (2);
        \draw [style={blue_edge}] (2) to (4);
        \draw [style={blue_edge}] (4) to (3);
        \draw [style={blue_edge}] (4) to (0);
        \draw [style={blue_edge}] (4) to (1);
        \draw [style={red_edge}] (0) to (5);
        \draw [style={red_edge}] (5) to (1);
        \draw [style={red_edge}] (5) to (3);
        \draw [style={red_edge}] (5) to (4);
        \draw [style={red_edge}] (5) to (2);
        \draw [style={blue_edge}] (2) to (6);
        \draw [style={blue_edge}] (6) to (4);
        \draw [style={blue_edge}] (3) to (6);
        \draw [style={blue_edge}] (5) to (6);
        \draw [style={blue_edge}] (6) to (0);
        \draw [style={blue_edge}] (1) to (6);
        \draw [style={red_edge}] (7) to (4);
        \draw [style={red_edge}] (7) to (2);
        \draw [style={red_edge}] (1) to (7);
        \draw [style={red_edge}] (7) to (0);
        \draw [style={red_edge}] (3) to (7);
        \draw [style={red_edge}] (7) to (5);
        \draw [style={red_edge}] (6) to (7);
        \draw [style={red_edge}] (0) to (8);
        \draw [style={red_edge}] (0) to (9);
        \draw [style={red_edge}] (0) to (10);
        \end{pgfonlayer}
        \end{tikzpicture}%
        \begin{tikzpicture}[scale=0.8]
        \begin{pgfonlayer}{nodelayer}
        \node [style={black_dot}, label={[xshift=-19pt, yshift=-8pt]:$\color{black}{(}\color{blue}{2}\color{black}{,}\color{red}{8}\color{black}{)}$}] (0) at (90.00:1.41) {};
        \node [style={black_dot}, label={left:$\color{black}{(}\color{blue}{4}\color{black}{,}\color{red}{3}\color{black}{)}$}] (1) at (135.00:1.41) {};
        \node [style={black_dot}, label={left:$\color{black}{(}\color{blue}{3}\color{black}{,}\color{red}{4}\color{black}{)}$}] (2) at (180.00:1.41) {};
        \node [style={black_dot}, label={below:$\color{black}{(}\color{blue}{3}\color{black}{,}\color{red}{5}\color{black}{)}$}] (3) at (225.00:1.41) {};
        \node [style={black_dot}, label={below:$\color{black}{(}\color{blue}{5}\color{black}{,}\color{red}{2}\color{black}{)}$}] (4) at (270.00:1.41) {};
        \node [style={black_dot}, label={below:$\color{black}{(}\color{blue}{1}\color{black}{,}\color{red}{6}\color{black}{)}$}] (5) at (315.00:1.41) {};
        \node [style={black_dot}, label={right:$\color{black}{(}\color{blue}{6}\color{black}{,}\color{red}{1}\color{black}{)}$}] (6) at (360.00:1.41) {};
        \node [style={black_dot}, label={right:$\color{black}{(}\color{blue}{0}\color{black}{,}\color{red}{7}\color{black}{)}$}] (7) at (405.00:1.41) {};
        \node [style={black_dot}, label={right:$\color{black}{(}\color{blue}{0}\color{black}{,}\color{red}{1}\color{black}{)}$}] (8) at (80.00:2.41) {};
        \node [style={black_dot}, label={left:$\color{black}{(}\color{blue}{0}\color{black}{,}\color{red}{1}\color{black}{)}$}] (10) at (100.00:2.41) {};
        \end{pgfonlayer}
        \begin{pgfonlayer}{edgelayer}
        \draw [style={red_edge}] (1) to (0);
        \draw [style={red_edge}, bend right=15] (0) to (3);
        \draw [style={red_edge}] (0) to (2);
        \draw [style={red_edge}, bend right=15] (3) to (0);
        \draw [style={red_edge}] (2) to (3);
        \draw [style={blue_edge}] (3) to (1);
        \draw [style={blue_edge}] (1) to (2);
        \draw [style={blue_edge}] (2) to (4);
        \draw [style={blue_edge}] (4) to (3);
        \draw [style={blue_edge}] (4) to (0);
        \draw [style={blue_edge}] (4) to (1);
        \draw [style={red_edge}] (0) to (5);
        \draw [style={red_edge}] (5) to (1);
        \draw [style={red_edge}] (5) to (3);
        \draw [style={red_edge}] (5) to (4);
        \draw [style={red_edge}] (5) to (2);
        \draw [style={blue_edge}] (2) to (6);
        \draw [style={blue_edge}] (6) to (4);
        \draw [style={blue_edge}] (3) to (6);
        \draw [style={blue_edge}] (5) to (6);
        \draw [style={blue_edge}] (6) to (0);
        \draw [style={blue_edge}] (1) to (6);
        \draw [style={red_edge}] (7) to (4);
        \draw [style={red_edge}] (7) to (2);
        \draw [style={red_edge}] (1) to (7);
        \draw [style={red_edge}] (7) to (0);
        \draw [style={red_edge}] (3) to (7);
        \draw [style={red_edge}] (7) to (5);
        \draw [style={red_edge}] (6) to (7);
        \draw [style={red_edge}] (0) to (8);
        \draw [style={red_edge}] (0) to (10);
        \end{pgfonlayer}
        \end{tikzpicture}%
        \begin{tikzpicture}[scale=0.8]
        \begin{pgfonlayer}{nodelayer}
        \node [style={black_dot}, label={[xshift=-17pt, yshift=-6pt]$\color{black}{(}\color{blue}{3}\color{black}{,}\color{red}{6}\color{black}{)}$}] (1) at (90.00:1.41) {};
        \node [style={black_dot}, label={left:$\color{black}{(}\color{blue}{3}\color{black}{,}\color{red}{5}\color{black}{)}$}] (0) at (135.00:1.41) {};
        \node [style={black_dot}, label={left:$\color{black}{(}\color{blue}{4}\color{black}{,}\color{red}{3}\color{black}{)}$}] (2) at (180.00:1.41) {};
        \node [style={black_dot}, label={below:$\color{black}{(}\color{blue}{3}\color{black}{,}\color{red}{4}\color{black}{)}$}] (3) at (225.00:1.41) {};
        \node [style={black_dot}, label={below:$\color{black}{(}\color{blue}{5}\color{black}{,}\color{red}{2}\color{black}{)}$}] (4) at (270.00:1.41) {};
        \node [style={black_dot}, label={below:$\color{black}{(}\color{blue}{2}\color{black}{,}\color{red}{5}\color{black}{)}$}] (5) at (315.00:1.41) {};
        \node [style={black_dot}, label={right:$\color{black}{(}\color{blue}{6}\color{black}{,}\color{red}{1}\color{black}{)}$}] (6) at (360.00:1.41) {};
        \node [style={black_dot}, label={right:$\color{black}{(}\color{blue}{0}\color{black}{,}\color{red}{7}\color{black}{)}$}] (7) at (405.00:1.41) {};
        \node [style={black_dot}, label={right:$\color{black}{(}\color{blue}{0}\color{black}{,}\color{red}{1}\color{black}{)}$}] (10) at (90:2.41) {};
        \end{pgfonlayer}
        \begin{pgfonlayer}{edgelayer}
        \draw [style={red_edge}, bend right=15] (0) to (1);
        \draw [style={red_edge}, bend right=15] (1) to (0);
        \draw [style={blue_edge}] (2) to (1);
        \draw [style={red_edge}] (0) to (3);
        \draw [style={red_edge}] (3) to (1);
        \draw [style={blue_edge}] (2) to (3);
        \draw [style={red_edge}] (2) to (0);
        \draw [style={blue_edge}] (4) to (3);
        \draw [style={blue_edge}] (2) to (4);
        \draw [style={blue_edge}] (4) to (0);
        \draw [style={blue_edge}] (4) to (1);
        \draw [style={blue_edge}] (0) to (5);
        \draw [style={red_edge}] (5) to (2);
        \draw [style={red_edge}] (5) to (4);
        \draw [style={red_edge}] (5) to (3);
        \draw [style={red_edge}] (5) to (1);
        \draw [style={blue_edge}] (6) to (5);
        \draw [style={blue_edge}] (4) to (6);
        \draw [style={blue_edge}] (6) to (2);
        \draw [style={blue_edge}] (6) to (3);
        \draw [style={blue_edge}] (6) to (1);
        \draw [style={blue_edge}] (6) to (0);
        \draw [style={red_edge}] (7) to (5);
        \draw [style={red_edge}] (6) to (7);
        \draw [style={red_edge}] (7) to (3);
        \draw [style={red_edge}] (7) to (4);
        \draw [style={red_edge}] (2) to (7);
        \draw [style={red_edge}] (7) to (1);
        \draw [style={red_edge}] (0) to (7);
        \draw [style={red_edge}] (10) to (1);
        \end{pgfonlayer}
        \end{tikzpicture}
        \caption{2-liecs of the split graph $G$ when $n=8$, $d_1 \in \{ 1,2,3\}$ and $d_2 = 0$ from the proof of Theorem~\ref{theorem_split}.}
        \label{fig:split8}
    \end{figure}
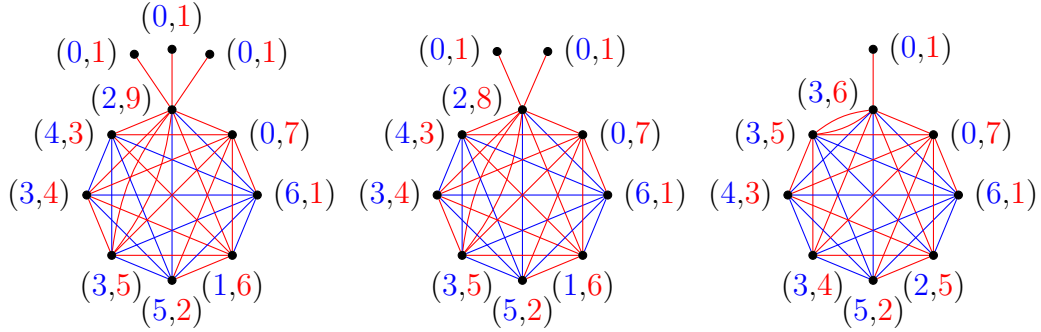
    
\end{proof}

\section{Powers of cycles}

In this section, we first present several notations that are specific to this part. 
Then we provide the general idea of the proof of the main result of this section -- Theorem~\ref{thm_powers_cycles}.
The proof is based on several lemmas, whose respective proofs are provided in the separate subsections of this section.

\subsection{Further notation and the general idea}

We start by defining almost irregular graphs, that are, in some sense, the opposite of regular graphs (see~\cite{Behzad Chartrand} and \cite{Chartrand}). 
While regular graphs have all their vertices of the same degree, almost irregular graphs are closest to having all vertices of distinct degrees.
Clearly, for $n \geq 2$, there is no graph of order $n$ with degree sequence $0,1,\dots, n-1$.
However, for each $n \geq 2$, there is a graph whose vertices have distinct degrees with the exception of a pair of vertices that have the same degrees.
There are some basic properties of almost irregular graphs, and it is a rather easy exercise to prove most of them. We list some of them in the following observation:
\begin{rem}\label{remark_AIG}
    Let $G$ be an almost irregular graph of order $n \geq 2$. 
    \begin{enumerate}
        \item If $G$ is connected then $G$ contains two vertices of degree $\left\lfloor \frac{n}{2} \right\rfloor$, and one vertex of each degree from $[1,n-1] \setminus \{ \left\lfloor \frac{n}{2} \right\rfloor \}$.
        \item  If $G$ is disconnected then $G$ contains two vertices of degree $\left\lfloor \frac{n-1}{2} \right\rfloor$, and one vertex of each degree from $[0,n-2] \setminus \{ \left\lfloor \frac{n-1}{2} \right\rfloor \}$.
        \item The complement of $G$ is an almost irregular graph of order $n$. 
        \item  There are exactly two non-isomorphic almost irregular graphs of order $n \geq 2$, one of them is connected and the second one is disconnected.
        \item\label{remark_AIG_con} If $G$ is connected, two vertices $u,v$ of $G$ are adjacent if and only if $\deg(u) + \deg(v) \geq n$.
        \item\label{remark_AIG_discon} If $G$ is disconnected, two vertices $u,v$ of $G$ are adjacent if and only if $\deg(u) + \deg(v) \geq n-1$.
    \end{enumerate}
\end{rem}

Almost irregular graphs are used in the respective proofs of this section as a base for constructing larger graphs.
For proving the main theorem of this section, it is important to be able to order the vertices of a graph, the vertices of an almost irregular graph in particular.
Hence, in some places, mainly in the constructive parts of the proofs, we use ordered graphs, i.e., graphs with a given ordering of vertices.
Such an ordering of vertices may be given explicitly or, more often, implicitly using the indices of vertices. 

From \ref{remark_AIG_con} and \ref{remark_AIG_discon} of Remark~\ref{remark_AIG} we have the following.
\begin{rem}\label{remark_ordering_degrees_AIG}
    The ordering of vertices of an ordered almost irregular graph $G$ is uniquely given by the ordering of elements of the multiset of vertices of $G$.
\end{rem}

When considering the orderings of vertices of almost irregular graphs, it is sufficient, according to Remark~\ref{remark_ordering_degrees_AIG}, to consider only the orderings of degrees of their vertices; and we do this using lists.
We recall some notation connected to lists.
By the length of a list $L$, we mean the number of (not necessarily unique) elements of $L$.
List $L'$ is a sublist of $L$ if it consists of some elements of $L$ and their order is preserved, i.e., if $L = l_1, \dots, l_n$ and $1 \leq n_1 \leq \cdots \leq n_k \leq n$, then $L'= l_{n_1}, \dots, l_{n_k}$ is a sublist of $L$. 
If $L_1 = l_1, \dots, l_k$ and $L_2 = l_{k+1}, \dots, l_m$ are two lists, the list $L_1 L_2 = l_1, \dots, l_m$ is a concatenation of $L_1$ and $L_2$.

Moreover, for the sake of shorter formulae, we use $[r,s]$ to denote the integer interval with bounds $r$ and $s$; i.e., $[r,s] = \{r, \dots, s\}$ if $r \leq s$ and $[r,s] = \emptyset$ if $r>s$.

As was partially mentioned before, we use almost irregular graphs as a base of building blocks from which we create a locally irregular subgraph $G$ of the $k$-th power of $C_n$, such that $C^k_n - G$ is also locally irregular.
To precisely define one of such building blocks, we use the following definition of $A(t,k)$. Note that, half-edges, i.e., edges incident to only one vertex, are used in this definition.

Let $t$ and $k$ be integers such that $t \geq k \geq 2$. 
By $A(t,k)$ we denote a graph with vertices $a_0, \dots, a_{t+1}$ which satisfies the following conditions \ref{con_a1}--\ref{con_a6}:
\begin{enumerate}[label=(a\arabic*)]
    \item\label{con_a1} $\deg_{A(t,k)}(a_0) + \deg_{A(t,k)}(a_{t+1}) \geq t+1$.
    \item\label{con_a2} $\{\deg_{A(t,k)}(a_i) \colon i \in [1, t]\} = [1,t]$.
    \item\label{con_a3} $a_0$ and $a_{t+1}$ are not incident to any half-edge, and each other vertex is incident to at most one half-edge.
    \item\label{con_a4} There is $\ell \in [0,\left\lfloor \frac{t}{2} \right\rfloor]$ such that $\{a_i \colon i \in [1, \ell] \cup [t-\ell+1, t]\}$ is the set of all vertices of $A(t,k)$ incident to half-edges.
    \item\label{con_a5} If $a_ia_j \in E(A(t,k))$ then $|i-j| \leq k$.
    \item\label{con_a6} If $a_ia_j \in E(A(t,k))$ and $i < t-\ell + 1 \leq j$ then $j - i \leq k - 1$.
\end{enumerate}

Note that graphs satisfying \ref{con_a1}--\ref{con_a6} may not be unique for fixed $k$ and $t$, and we will be interested only in the existence of such graphs. 
If there is a graph $A(t,k)$ for some parameters $t$ and $k$, we can use several copies of it to prove the following:
\begin{lem}\label{lemma_A(t,k)}
    Let $k$ and $t$ be integers such that $\frac{8k-5}{5} \geq t \geq k \geq 2$. 
    If there is a graph $A(t,k)$
    then $\mathrm{lir}(C_{p(t+1)+q(t+2)}^k) = 2$ for each nonnegative integers $p$ and $q$, such that $p+q \geq 2$.
\end{lem}

In the proofs of the following two lemmas, the constructions of $A(t,k)$ for particular values of $t$ and $k$ are described, hence showing the existence of them. 
Then Lemma~\ref{lemma_A(t,k)} is used.
\begin{lem}\label{lemma_no_half_edges}
    Let $k,t,p,q$ be nonnegative integers such that $t \geq k \geq 2$, $p+q \geq 2$, 
    and
    \begin{align}\label{bound_t_fork}
        t \leq \left\{\begin{matrix*}[l]
            \frac{4k-2}{3} & \text{if } t \text{ is even,}\\
            \frac{4k-3}{3} & \text{if } t \text{ is odd.}
        \end{matrix*}\right.
    \end{align}
    Then $\mathrm{lir}(C_{p(t+1) + q(t+2)}^k) = 2$.
\end{lem}

\begin{lem}\label{lemma_half_edges}
    Let $k,t,p,q$ be nonnegative integers such that $k \geq 4$, $\frac{4k-1}{3} \leq t \leq \frac{8k-5}{5}$, and $p+q \geq 2$.
    Then $\mathrm{lir}(C_{p(t+1) + q(t+2)}) = 2$.
\end{lem}


Using Lemma~\ref{lemma_no_half_edges} and Lemma~\ref{lemma_half_edges}, and solving a small number of particular cases, we prove the main theorem of this section:
\begin{tw}\label{thm_powers_cycles}
    $\mathrm{lir}(C_n^k) = 2$ for each $k \geq2$ and $n \geq 2k+2$.
\end{tw}
\begin{proof}
    In the following, we heavily use the fact that for coprimes $a$ and $b$, every integer $n \geq (a-1)(b-1)$ can be expressed in the form $n = pa + qb$ where $p$ and $q$ are nonnegative integers (see coin problem~\cite{Sylvester}).
    In particular, for $k = 2$, Lemma~\ref{lemma_no_half_edges} implies that $\mathrm{lir}(C_{3p+4q}^2) = 2$ (for $t=2$). 
    Since every integer $n \geq 6$ can be expressed in the form $n = 3p + 4q$ for some nonnegative integers $p$ and $q$, $p+q \geq 2$, we get that $\mathrm{lir}(C_n^2) = 2$ for $n \geq 6$.

    In the case when $k=3$, Lemma~\ref{lemma_no_half_edges} yields that $\mathrm{lir}(C_{4p+5q}^3) = 2$ (for $t = 3$) for each pair of nonnegative integers $p$ and $q$ such that $p+q \geq 2$.
    Each $n \geq 12$ can be expressed in a form $n = 4p + 5q$ for suitable integers $p$ and $q$.
    Moreover, it is easy to check that also each $n \in \{8,9,10\}$ can be expressed in the same form.
    For $n=11$, the decomposition of $C_{11}^3$ is explicitly provided in Figure~\ref{fig:C_11^3_decomp}, which completes the proof of the theorem in the case when $k = 3$.

    \begin{figure}
        \centering
        \begin{tikzpicture}
        	\begin{pgfonlayer}{nodelayer}
        		\node [style={black_dot}, label={above:$\mathcolor{black}{(} \mathcolor{blue}{5} \mathcolor{black}{,} \mathcolor{red}{1} \mathcolor{black}{)}$}] (0) at (0, 1.5) {};
        		\node [style={black_dot}, label={135:$\mathcolor{black}{(} \mathcolor{blue}{0} \mathcolor{black}{,} \mathcolor{red}{6} \mathcolor{black}{)}$}] (1) at (-0.75, 1.25) {};
        		\node [style={black_dot}, label={left:$\mathcolor{black}{(} \mathcolor{blue}{2} \mathcolor{black}{,} \mathcolor{red}{4} \mathcolor{black}{)}$}] (2) at (-1.25, 0.75) {};
        		\node [style={black_dot}, label={left:$\mathcolor{black}{(} \mathcolor{blue}{4} \mathcolor{black}{,} \mathcolor{red}{2} \mathcolor{black}{)}$}] (3) at (-1.5, 0) {};
        		\node [style={black_dot}, label={left:$\mathcolor{black}{(} \mathcolor{blue}{1} \mathcolor{black}{,} \mathcolor{red}{5} \mathcolor{black}{)}$}] (4) at (-1.25, -0.75) {};
        		\node [style={black_dot}, label={270:$\mathcolor{black}{(} \mathcolor{blue}{3} \mathcolor{black}{,} \mathcolor{red}{3} \mathcolor{black}{)}$}] (5) at (-0.5, -1.25) {};
        		\node [style={black_dot}, label={270:$\mathcolor{black}{(} \mathcolor{blue}{5} \mathcolor{black}{,} \mathcolor{red}{1} \mathcolor{black}{)}$}] (6) at (0.5, -1.25) {};
        		\node [style={black_dot}, label={right:$\mathcolor{black}{(} \mathcolor{blue}{0} \mathcolor{black}{,} \mathcolor{red}{6} \mathcolor{black}{)}$}] (7) at (1.25, -0.75) {};
        		\node [style={black_dot}, label={right:$\mathcolor{black}{(} \mathcolor{blue}{4} \mathcolor{black}{,} \mathcolor{red}{2} \mathcolor{black}{)}$}] (8) at (1.5, 0) {};
        		\node [style={black_dot}, label={right:$\mathcolor{black}{(} \mathcolor{blue}{3} \mathcolor{black}{,} \mathcolor{red}{3} \mathcolor{black}{)}$}] (9) at (1.25, 0.75) {};
        		\node [style={black_dot}, label={45:$\mathcolor{black}{(} \mathcolor{blue}{1} \mathcolor{black}{,} \mathcolor{red}{5} \mathcolor{black}{)}$}] (10) at (0.75, 1.25) {};
        	\end{pgfonlayer}
        	\begin{pgfonlayer}{edgelayer}
        		\draw [style={blue_edge}] (0) to (2);
        		\draw [style={blue_edge}] (2) to (3);
        		\draw [style={blue_edge}] (3) to (0);
        		\draw [style={blue_edge}] (0) to (10);
        		\draw [style={blue_edge}] (0) to (9);
        		\draw [style={blue_edge}] (0) to (8);
        		\draw [style={blue_edge}] (9) to (6);
        		\draw [style={blue_edge}] (6) to (7);
        		\draw [style={blue_edge}] (7) to (8);
        		\draw [style={blue_edge}] (5) to (8);
        		\draw [style={blue_edge}] (5) to (6);
        		\draw [style={blue_edge}] (3) to (6);
        		\draw [style={blue_edge}] (5) to (4);
        		\draw [style={blue_edge}] (4) to (6);
        		\draw [style={red_edge}] (1) to (0);
        		\draw [style={red_edge}] (1) to (2);
        		\draw [style={red_edge}] (1) to (3);
        		\draw [style={red_edge}] (1) to (4);
        		\draw [style={red_edge}] (1) to (10);
        		\draw [style={red_edge}] (1) to (9);
        		\draw [style={red_edge}] (2) to (10);
        		\draw [style={red_edge}] (2) to (4);
        		\draw [style={red_edge}] (2) to (5);
        		\draw [style={red_edge}] (3) to (4);
        		\draw [style={red_edge}] (3) to (5);
        		\draw [style={red_edge}] (4) to (7);
        		\draw [style={red_edge}] (5) to (7);
        		\draw [style={red_edge}] (6) to (8);
        		\draw [style={red_edge}] (8) to (9);
        		\draw [style={red_edge}] (8) to (10);
        		\draw [style={red_edge}] (9) to (10);
        		\draw [style={red_edge}] (10) to (7);
        		\draw [style={red_edge}] (7) to (9);
        	\end{pgfonlayer}
        \end{tikzpicture}
        \caption{2-liec of $C_{11}^3$.}
        \label{fig:C_11^3_decomp}
    \end{figure}
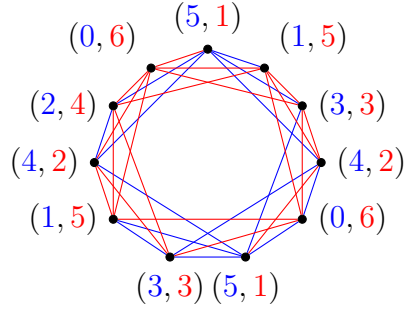

    Hence, in the following, assume that $k \geq 4$; in this case we can use Lemma~\ref{lemma_half_edges} on top of Lemma~\ref{lemma_no_half_edges}.

    Let $t$ be an integer, $k \leq t \leq \frac{8k-5}{5}$. 
    If $t$ is even, then $t$ satisfies $t \leq \frac{4k-2}{3}$ or $t \geq \frac{4k-1}{3}$.
    If $t$ is odd, then $t$ satisfies $t \leq \frac{4k-3}{3}$ or $t \geq \frac{4k-2}{3}$.
    However, if $t$ is odd then clearly $t \neq \frac{4k-2}{3}$ (as $3t$ is odd and $4k-2$ is even).
    Hence, $t \geq \frac{4k-2}{3}$ implies $t \geq \frac{4k-1}{3}$ in the case when $t$ is odd.
    It follows from these simple observations, Lemma~\ref{lemma_no_half_edges}, and Lemma~\ref{lemma_half_edges} that $\mathrm{lir}(C_{p(t+1)+q(t+2)}) = 2$ for every nonnegative integers $t$, $p$, and $q$ such that $k \leq t \leq \frac{8k-5}{5}$ and $p + q \geq 2$.

    Hence, in the following, we focus on showing that each $n \geq 2k + 2$ can be written in the form $n=p(t+1) + q(t+2)$ for some nonnegative integers $p$ and $q$, $p + q \geq 2$, and some $t$ satisfying $k \leq t \leq \frac{8k-5}{5}$.

    Let $I(r,t) = \{p(t+1)+q(t+2) \colon p+q=r\}$.
    Observe that $p(t+1)+q(t+2)$ and $(p-1)(t+1) + (q+1)(t+2)$ are consecutive integers. 
    Thus, $I(r,t)$ is an integer interval for each $r$ and $t$.
    
    Moreover, $r(t+2) \in I(r,t) \cap I(r, t+1)$, and 
    thus $\bigcup\limits_{t=k}^{\lfloor \frac{8k-5}{5}\rfloor} I(r,t)$ is an integer interval for fixed $r$.
    To prove the theorem, we show that for each $k \geq 4$ the set $U(k) = \bigcup\limits_{r \geq 2}\bigcup\limits_{t=k}^{\lfloor \frac{8k-5}{5}\rfloor} I(r,t)$ contains all integers greater or equal to $2k+2$.

    Observe that $\min I(2,k) = 2(k+1)$, hence, $2k+2 \in U(k)$.
    Using the fact that $\bigcup\limits_{t=k}^{\lfloor \frac{8k-5}{5}\rfloor} I(r,t)$ is an integer interval for fixed $r$, we get that $U(k)$ does not contain all integers greater or equal to $2k+2$ only in the case when there is an integer between the maximum of $I(r,\lfloor\frac{8k-5}{5}\rfloor)$ and the minimum of $I(r+1,k)$.
    Equivalently, if 
    \begin{align}\label{max_min_intervals}
        1 + \max I(r,\left\lfloor \frac{8k-5}{5}\right\rfloor) \geq \min I(r+1,k)   
    \end{align}
    then $I(r,\lfloor \frac{8k-5}{5}\rfloor) \cup I(r+1,k)$ is an integer interval.
    We have that $1 + \max I(r,\lfloor \frac{8k-5}{5}\rfloor) \geq 1 + r\frac{8k-9}{5}$ and $\min I(r+1,k) = (r+1)(k+1)$.
    Hence, if 
    \begin{align}\label{max_min_intervals_2}
        8kr-9r \geq 5rk + 5r + 5k
    \end{align}
    then \eqref{max_min_intervals} holds.
    Moreover, \eqref{max_min_intervals_2} is equivalent to 
    \begin{align*}
        r(3k - 4) - 5k \geq 0,
    \end{align*}
    which is true for every $r \geq 3$ and every $k \geq 4$, and for $r=2$ and $k \geq 8$.
    Therefore, for $k \geq 8$, the set $U(k)$ contains all integers greater or equal to $2k + 2$, 
    and for $k \in {4,5,6,7}$, $U(k)$ contains all integers greater or equal to $2k+2$ with the possible exception of integers lying between $\max I(2,\lfloor \frac{8k-5}{5} \rfloor)$ and $\min I(3,k)$.

    For $k=4$ we have $\max I(2,\lfloor \frac{8k-5}{5} \rfloor) = 2\cdot(5+2) = 14 $ and $\min I(3,k) = 3 \cdot (4 + 1) = 15$, hence $U(4)$ contains all integers greater or equal to $10$.
    For $k=5$ we have $\max I(2,\lfloor \frac{8k-5}{5} \rfloor) = 2\cdot(7+2) = 18 $ and $\min I(3,k) = 3 \cdot (5 + 1) = 18$, hence $U(5)$ contains all integers greater or equal to $12$.
    For $k=6$ we have $\max I(2,\lfloor \frac{8k-5}{5} \rfloor) = 2\cdot(8+2) = 20 $ and $\min I(3,k) = 3 \cdot (6 + 1) = 21$, hence $U(6)$ contains all integers greater or equal to $14$.
    For $k=7$ we have $\max I(2,\lfloor \frac{8k-5}{5} \rfloor) = 2\cdot(10+2) = 24 $ and $\min I(3,k) = 3 \cdot (7 + 1) = 24$, hence $U(7)$ contains all integers greater or equal to $16$.
    This completes the proof of the theorem.
\end{proof}
Note that the bound on $n$ in Theorem~\ref{thm_powers_cycles} is sharp, since $C_n^k$ for $n < 2k+2$ is a complete graph, and there is no 2-liec of any complete graphs of order at least four~\cite{Baudon Bensmail Przybylo Wozniak}.

\subsection{Proof of Lemma~\ref{lemma_A(t,k)}}

    Throughout the proof of this lemma, assume that $A(t,k)$ is fixed.
    Let $B(t,k)$ be the ordered graph with vertices $b_0, \dots, b_{t+2}$ which is obtained from $A(t,k)$ by adding a new vertex $a_{t+2}$ of degree zero, after relabelling vertices in such a way that 
    $b_i = a_i$ for $i \in [0,t-\ell]$, 
    $b_{t-\ell+1} = a_{t+2}$, 
    and $b_{i+1} = a_{i}$ for $i \in [t-\ell+1, t+1]$. 
    From the properties (a1) -- (a6) of $A(t,k)$ and the construction of $B(t,k)$, the properties \ref{con_b1}--\ref{con_b5} of $B(t,k)$ follows:
    \begin{enumerate}[label=(b\arabic*)]
        \item\label{con_b1} $\deg_{B(t,k)}(b_0) + \deg_{B(t,k)}(b_{t+2}) \geq t+1$.
        \item\label{con_b2} $\{\deg_{B(t,k)}(b_i) \colon i \in [1, t+1]\} = [0,t]$.
        \item\label{con_b3} $b_0$ and $b_{t+2}$ are not incident to any half-edge, and each other vertex is incident to at most one half-edge.
        \item\label{con_b4} $\{b_i \colon i \in [1, \ell] \cup [t-\ell+2, t+1]\}$ is the set of all vertices of $B(t,k)$ incident to half-edges.
        \item\label{con_b5} If $b_ib_j \in E(B(t,k))$ then $|i-j| \leq k$.
    \end{enumerate}

    Note that the property~\ref{con_b5} follows straight from properties~\ref{con_a5} and \ref{con_a6}, and the fact that a new vertex of degree 0 was added right before the last $\ell + 1$ vertices.

    Now, consider $p$ copies of $A(t,k)$, denoted by $A_1, \dots, A_p$, and $q$ copies of $B(t,k)$, denoted by $B_1, \dots, B_q$.
    By $a_i^j$ we will denote the vertex $a_i$ of $A_j$, 
    and by $b_i^j$ we denote the vertex $b_i$ in $B_j$.
    
    If $p=0$, let $G$ be a graph obtained from $B_1, \dots, B_k$ by identifying $b_{t+2}^j$ with $b_0^{j+1}$, and replacing half-edges incident to $b_{t-\ell+2}^j, \dots, b_{t+1}^j$ and $b_1^{j+1}, \dots, b_{\ell}^{j+1}$ with edges $b_{t-\ell +1 + i}^jb_{i}^{j+1}$, for each $j \in \mathbb{Z}_q$, and each $i \in [1,\ell]$.
    
    If $q = 0$, let $G$ be a graph obtained from $A_1, \dots, A_p$ by identifying $a_{t+1}^j$ with $a_0^{j+1}$ and replacing half-edges incident to $a_{t-\ell+1}^j, \dots, a_t^j$ and $a_1^{j+1}, \dots, a_{\ell}^{j+1}$ with edges $a_{t-\ell + i}^ja_{i}^{j+1}$, for each $j \in \mathbb{Z}_p$, and each $i \in [1,\ell]$.

    If $p \neq 0$ and $q \neq 0$, let $G$ be a graph which is obtained from $A_1, \dots, A_p, B_1, \dots, B_q$ after:
    \begin{itemize}
        \item identifying $a_{t+1}^j$ with $a_0^{j+1}$, and replacing half-edges incident to $a_{t-\ell+i}^j$ and $a_i^{j+1}$ with the edge $a_{t-\ell+i}^j a_i^{j+1}$ for each $i \in [1,\ell]$ and $j \in [1,p-1]$,
        \item identifying $a_{t+1}^p$ with $b_0^1$, and replacing half-edges incident to $a_{t-\ell+i}^p$ and $b_i^1$ with the edge $a_{t-\ell+i}^p b_i^1$ for each $i \in [1,\ell]$,
        \item identifying $b_{t+2}^j$ with $b_0^{j+1}$, and replacing half-edges incident to $b_{t+1-\ell+i}^j$ and $b_i^{j+1}$ with the edge $b_{t+1-\ell+i}^j b_i^{j+1}$ for each $i \in [1,\ell]$ and $j \in [1,q-1]$,
        \item identifying $b_{t+2}^q$ with $a_0^1$, and replacing half-edges incident to $b_{t+1-\ell+i}^q$ and $a_i^1$ with the edge $b_{t+1-\ell+i}^q a_i^1$ for each $i \in [1,\ell]$.
    \end{itemize}
    Let the ordering of the vertices of $G$ be given by the list 
    $$L=a_1^1, \dots, a_{t+1}^1, a_1^2, \dots, a_{t+1}^2, \dots, a_{t+1}^p, b_1^1, \dots, b_{t+2}^1, b_1^2, \dots, b_{t+2}^2, \dots, b_{t+2}^q.$$
    Clearly, $G$ has $p(t+1) + q(t+2)$ vertices. 

    Let $C$ be the cycle of length $p(t+1) + q(t+2)$ with $V(C) = V(G)$ and $uv \in E(C)$ whenever $u$ and $v$ are consecutive in the ordering of $V(G)$, or $u$ and $v$ are the first and the last vertex in the ordering, respectively.
    We want to show that $G$ is a subgraph of $C^k$ and that both $G$ and $C^k-G$ are locally irregular.
    Observe that, to show this, it is enough to show that two vertices adjacent in $G$ are at distance at most $k$ in $C$ (hence, $G$ is a subgraph of $C^k$),
    and that two vertices of $G$ with the same degrees are at distance at least $k+1$ in $C$.
    To denote the distance of two vertices $u$ and $v$ in $C$ we will use the notion $\mathrm{dist}_C(u,v)$.

    We first show that $uv \in E(G)$ implies $\mathrm{dist}_C(u,v) \leq k$.
    This is clearly true if both $u$ and $v$ are vertices of the same copy of $A(t,k)$ or $B(t,k)$, see \ref{con_a5}, \ref{con_a6}, and \ref{con_b5}.
    Moreover, if an edge joining two vertices from different copies of $A(t,k)$ and/or $B(t,k)$ is added in the process of creating $G$, such an edge replaces two half-edges incident to vertices whose distance in $C$ is at most $\ell$ (this is easy to see from the construction of $G$, due to the indices of the vertices). 
    Since $\ell \leq \left\lfloor\frac{t}{2}\right\rfloor$, and $t \leq \frac{8k-5}{5}$, we get that the distance of such two vertices in $C$ is at most $k$.
    Hence, $G$ is a subgraph of $C^k$.

    Now, let $u$ and $v$ be two vertices of $G$ such that $\deg_G(u) = \deg_G(v)$.
    The following observation yields the fact that $\mathrm{dist}_C(u,v) \geq k+1$:
    If $q = 0$, then clearly $u = a_i^{j_1}$ and $v=a_i^{j_2}$ for some $i \in [1, t+1]$ and $j_1, j_2 \in \mathbb{Z}_p$, $j_1 \neq j_2$. Since each $A_j$ has exactly $t+1$ vertices, we get that $\mathrm{dist}_C(u,v) \geq t+1$.
    If $q \neq 0$ then $\mathrm{dist}_C(u,v) \geq t+1$, as there are potentially more vertices (namely those with degree zero in $G$) that are on the shortest $u,v$-path in $C$, when comparing to the case when $q=0$.
    It is clear from constructions of $B(t,k)$ and $G$, that the only vertices of degree zero in $G$ are vertices $b_{t - \ell}^j$ for each $j \in [1,q]$. 
    Hence, if $\deg_G(u) = \deg_G(v) = 0$ then $\mathrm{dist}_C(u,v) \geq t+2$.
    
    Thus, in each case, $\mathrm{dist}_C(u,v) \geq t+1 \geq k+1$ whenever $\deg_G(u) = \deg_G(v)$.
    This completes the proof that $G$ and $C^k - G$ are locally irregular subgraphs of $C^k$.
    To obtain a 2-liec of $C^k$, it is sufficient to color $G$ blue and $C^k-G$ red.

\subsection{Proof of Lemma~\ref{lemma_no_half_edges}}

    We show that there is a graph $A(t,k)$ with vertices $a_0, \dots, a_{t+1}$ with properties (a1) -- (a6).
    The result then follows from Lemma~\ref{lemma_A(t,k)}.

    Let $H_t$ be a disconnected almost irregular graph of order $t$.
    Let $M_t$ be the multiset of degrees of vertices of $H_t$, i.e., $M_t = \{0,1, \dots, \left\lfloor\frac{t-1}{2}\right\rfloor, \left\lfloor\frac{t-1}{2}\right\rfloor, \dots, t-2\}$.
    Let $L$ be a list of all $t$ elements of $M_t$ in which odd integers are listed first, in non-decreasing order, followed by even integers in non-increasing order.
    Denote by $l_i$ the $i$-th element of $L$, $i \in [1, |L|]$.

    \begin{claim}\label{claim_li_lj}
        If $l_i + l_j \geq t-1$ then $|i - j| \leq k$.
    \end{claim}
    \begin{proof}
        Consider the set $S = \{0, \dots, t-2\}$ and the list $L^*$ of all elements of $S$ in which odd integers are listed first, in increasing order, followed by even integers in decreasing order. 
        Clearly, $L^*$ is a sublist of $L$.
        Denote by $l_i^*$ the $i$-th element of $L^*$.
        Note that, to prove Claim~\ref{claim_li_lj}, it is sufficient to show that $l_i^* + l_j^* \geq t-1$ implies $|i-j|\leq k-1$, since there is at most one extra element between any two elements of $L^*$ in $L$ (namely the second occurrence of $\left\lfloor\frac{t-1}{2}\right\rfloor$).
    
        Note that, if $l_i^* \equiv l_j^* \pmod{2}$ for some $i, j \in \{0, \dots, t-2\}$, then $|i - j| \leq \frac{t}{2} - 1$.
        If $t$ is even, we have $\frac{t}{2} - 1 \leq \frac{2}{3}k - \frac{4}{3} \leq k - 1$ for $k \geq 2$.
        If, on the other hand, $t$ is odd, we have $\frac{t}{2} - 1 \leq \frac{2}{3}k - \frac{3}{2} \leq k-1$ for $k \geq 2$.
        Hence, in both cases, we obtain $|i-j| \leq k-1$.
        
        Now, let $i$ be given such that $l_i^*$ is odd.
        We calculate the maximum $j$ such that $l_i^* + l_j^* \geq t-1$.
        Since $l_i^* = 2i-1$, we have $l_j^* \geq t - 2i$.
        The case when $l_j^*$ is odd is solved in the previous paragraph, hence, assume that $l_j^*$ is even.
        Since $l_j^*$ is the $(t-j)$-th smallest even nonnegative integer, we get $2(t-j-1) = l_j^* \geq t-2i$, and from that $j-i \leq \frac{t-2}{2} \leq k-1$.
    \end{proof}

    \begin{claim}\label{claim_sublists_no_half_edges}
        There are lists $L_1$, $L_2$, and $L_3$ such that $L$ is a concatenation of them, i.e., $L = L_1 L_2 L_3$, such that $|L_1| + |L_3| = \left\lceil \frac{t}{2}\right\rceil$, and each element of the set $\{0, \dots, \left\lfloor\frac{t-1}{2}\right\rfloor\}$ is present in $L_1$ or $L_3$. 
        Moreover, if $t \equiv 0 \pmod{4}$ or $t \equiv 3 \pmod{4}$ then $|L_1| + |L_2| \leq k$ and $|L_2| + |L_3| \leq k-1$, and if $t \equiv 1 \pmod{4}$ or $t \equiv 2 \pmod{4}$ then $|L_1| + |L_2| \leq k-1$ and $|L_2| + |L_3| \leq k$.
    \end{claim}
    \begin{proof}
        Let $t = 4x + y$ for some nonnegative integers $x$ and $y$, $y \leq 3$.
        Let $L_1$ be the list of the first $x$ elements of $L$ if $y \in \{0,1,2\}$, and the list of the first $x+1$ elements of $L$ if $y =3$.
        Let $L_3$ be the list of the last $x$ elements of $L$ if $y = 0$, and the list of the last $x+1$ elements of $L$ if $y \in \{1,2,3\}$.
        Let $L_2$ be the sublist of $L$ such that $L=L_1L_2L_3$.
        For an overview of the lengths of the sublists see Table~\ref{table_almost_irregular}.

        Now, we will show that each element of the set $\{0, \dots, \left\lfloor\tfrac{t-1}{2}\right\rfloor\}$ is present in $L_1$ or $L_3$. 
        One should distinguish four cases depending on the value of $y$. 
        For each of such cases, a simple observation is needed. Hence, we provide it for a case when $y=0$, i.e., $t=4x$; for other cases, similar observations can be made, so we left it on the reader. 
        Suppose that $t=4x$. Then $L_1$ consists of $x$ smallest odd integers from $M_t$ ordered in a non-decreasing order. 
        Since the element $\left\lfloor\frac{t-1}{2}\right\rfloor$ which is twice in $M_t$ is $2x-1$, we have that $L_1 = 1, \dots, (2x-1)$. Similarly, $L_3$ consists of $x$ smallest even integers from $M_t$ ordered in the non-increasing order. Hence, $L_3=(2x-2), \dots, 0$. Clearly, every integer from $[0, \left\lfloor \tfrac{t-1}{2}\right\rfloor]$ is in $L_1$ or $L_3$.

        The fact that $|L_1| + |L_3| = \left\lceil \frac{t}{2}\right\rceil$ can be easily checked using Table~\ref{table_almost_irregular}. 

        Note that in Table~\ref{table_almost_irregular}, values of $k$ and $k-1$ are also provided. 
        These values follow from \eqref{bound_t_fork} and the fact that $k$ and $k-1$ are integers.
        In the case when $t$ is even, we get from \eqref{bound_t_fork} that $k \geq \frac{3t+2}{4}$. In the case when $t=4x$ we have $k \geq 3x+\frac{1}{2}$ and, since $k$ is integer, $k \geq 3x+1$.
        For $t = 4x+2$ we have $k \geq 3x + 2$.
        In the case when $t$ is odd, we get from \eqref{bound_t_fork} that $k \geq \frac{3t+3}{4}$. If $t = 4x+1$, we get $k \geq 3x + \frac{3}{2}$ and, since $k$ is integer, $k \geq 3x + 2$.
        If $t = 4x+3$, we have $k \geq 3x+3$.

       From Table~\ref{table_almost_irregular}, it is easy to see that if $t \equiv 0 \pmod{4}$ or $t \equiv 3 \pmod{4}$ then $|L_1| + |L_2| \leq k$ and $|L_2| + |L_3| \leq k-1$, and if $t \equiv 1 \pmod{4}$ or $t \equiv 2 \pmod{4}$ then $|L_1| + |L_2| \leq k-1$ and $|L_2| + |L_3| \leq k$. 

        
        \begin{table}[h]
            \centering
            \begin{tabular}{c|c|c|c|c|c|c}
            $t$     & $|L_1|$ & $|L_2|$ & $|L_3|$ & $k$   & $k-1$ & $\lfloor (t-1)/2 \rfloor$  \\ \hline
            $4x$    & $x$     & $2x$    & $x$     & $\geq 3x+1$    & $\geq 3x$ & $2x-1$ \\
            $4x+1$   & $x$     & $2x$  & $x+1$     & $\geq 3x + 2$    & $\geq 3x+1$ & $2x$     \\
            $4x+2$  & $x$   & $2x+1$  & $x+1$     & $\geq 3x + 2$ & $\geq 3x + 1$ & $2x$ \\
            $4x+3$  & $x+1$   & $2x+1$  & $x+1$     & $\geq 3x +3$  & $\geq 3x + 2$ & $2x+1$
            \end{tabular}
            \caption{Values considered in the proof of Claim~\ref{claim_sublists_no_half_edges}.}\label{table_almost_irregular}
        \end{table}
    \end{proof}

    In the following, let $L=L_1L_2L_3$ where $L_1$, $L_2$, and $L_3$ are sublists of $L$ with properties stated in the Claim~\ref{claim_sublists_no_half_edges}.
    Note that, since $|L_1|+|L_3| = \left\lceil \frac{t}{2}\right\rceil$, we have that $|L_2| = \left\lfloor \frac{t}{2}\right\rfloor$.
    Moreover, since each element of $[0, \left\lfloor\frac{t-1}{2} \right\rfloor]$ is in $L_1$ or $L_3$, we get that $L_2$ consists of elements of $[\left\lfloor\frac{t-1}{2} \right\rfloor, t-2]$.
    We now show a construction of $A(t,k)$. 
    We distinguish two cases.

    \textbf{Case 1.} Suppose that $t \equiv 0 \pmod{4}$ or $t \equiv 3 \pmod{4}$.
    In this case, according to Claim~\ref{claim_sublists_no_half_edges}, 
    \begin{align}\label{condition_Li}
        |L_1| + |L_2| \leq k \quad \text{and} \quad |L_2| + |L_3| \leq k-1.
    \end{align}
    Denote by $a_1, \dots, a_t$ the vertices of $H_t$ in such a way that $\deg_{H_t}(a_i) = l_i$.
    Let $A(t,k)$ be an ordered graph obtained from $H_t$ by adding new vertices $a_0$ and $a_{t+1}$, edges $a_0a_i$ for each $i \in [1,|L_1|+|L_2|]$, and edges $a_ja_{t+1}$ for each $j \in [|L_1|+1, |L|]$. 
    Clearly, $\deg_{A(t,k)}(a_0) + \deg_{A(t,k)}(a_{t+1}) = |L_1|+2|L_2|+|L_3| = \left\lceil \frac{t}{2}\right\rceil + 2\left\lfloor \frac{t}{2}\right\rfloor \geq t+1$. Hence $\ref{con_a1}$ is satisfied.
    
    If $1 \leq i \leq |L_1|$ or $|L_1|+|L_2| + 1 \leq i \leq |L|$ then $\deg_{A(t,k)}(a_i) = \deg_{H_t}(a_i) + 1 = l_i+1$.
    Since, $L_1L_3$ consists of elements of $[0,\left\lfloor\frac{t-1}{2}\right\rfloor]$, we get 
    \begin{align}\label{small_degrees}
        \left\{ \deg_{A(t,k)}(a_i) \colon i \in [1, |L_1|] \cup [|L_1|+|L_2| + 1, |L|]\right\} = [1, \dots, \left\lfloor\tfrac{t-1}{2}\right\rfloor + 1].
    \end{align}
    For $i \in [|L_1| + 1, |L_1| + |L_2|]$ we have $\deg_{A(t,k)}(a_i) = \deg_{H_t}(a_i) + 2 = l_i+2$.
    Since $L_2$ consists of elements of $[\left\lfloor\frac{t-1}{2} \right\rfloor, t-2]$, we have
    \begin{align}\label{large_degrees}
        \left\{ \deg_{A(t,k)}(a_i) \colon i \in [|L_1|+1, |L_1|+|L_2|]\right\} = [\left\lfloor\tfrac{t-1}{2}\right\rfloor + 2, t].
    \end{align}
    From \eqref{small_degrees} and \eqref{large_degrees} we have that \ref{con_a2} is satisfied.

    Since $A(t,k)$ in our case does not have half-edges, \ref{con_a3} and \ref{con_a4} are satisfied, and $\ell = 0$.
    Conditions \ref{con_a5} and \ref{con_a6} are satisfied due to Claim~\ref{claim_li_lj} and~\eqref{condition_Li}.
    
    Hence, $A(t,k)$ satisfies conditions \ref{con_a1}--\ref{con_a6}, and from Lemma~\ref{lemma_A(t,k)} we have that ${\rm lir}(C_{p(t+1)+q(t+2)}^k) = 2$ for each nonnegative integers $p$ and $q$, such that $p+q \geq 2$.

    \textbf{Case 2.} Suppose that $t \equiv 1 \pmod{4}$ or $t \equiv 2 \pmod{4}$.
    Then
    \begin{align}\label{condition_Li2}
        |L_1| + |L_2| \leq k-1 \quad \text{and} \quad |L_2| + |L_3| \leq k.
    \end{align}
    Denote by $a_1, \dots, a_t$ the vertices of $H_t$ in such a way that $\deg_{H_t}(a_i) = l_{t-i+1}$.
    Let $A(t,k)$ be an ordered graph obtained from $H_t$ by adding new vertices $a_0$ and $a_{t+1}$, edges $a_0a_i$ for each $i \in [1,|L_2|+|L_3|]$, and edges $a_ja_{t+1}$ for each $j \in [|L_3|+1, |L|]$. 
    Clearly, $\deg_{A(t,k)}(a_0) + \deg_{A(t,k)}(a_{t+1}) = |L_3|+2|L_2|+|L_1| = \left\lceil \frac{t}{2}\right\rceil + 2\left\lfloor \frac{t}{2}\right\rfloor \geq t+1$. Hence $\ref{con_a1}$ is satisfied.

    Using the facts that $t = |L|$,  and each element of $[1, \left\lfloor \tfrac{t-1}{2}\right\rfloor]$ is present in $L_1$ or $L_3$, we get
    \begin{align}\label{small_degrees_2}
        \begin{aligned}
            &\left\{ \deg_{A(t,k)}(a_i) \colon i \in [1, |L_3|] \cup [|L_2|+|L_3| + 1, |L|]\right\} \\
            &= \left\{ \deg_{H_t}(a_i) + 1 \colon i \in [1, |L_3|] \cup [|L_2|+|L_3| + 1, |L|]\right\} \\
            &=  \left\{ l_{j} + 1 \colon j \in [1, |L_1|] \cup [|L_1|+|L_2| + 1, |L|]\right\} \\
            &= [1,\left\lfloor \tfrac{t-1}{2}\right\rfloor+1].
        \end{aligned}
    \end{align}
    and similarly
    \begin{align}\label{large_degrees_2}
        \begin{aligned}
            &\left\{ \deg_{A(t,k)}(a_i) \colon i \in [|L_3|+1, |L_2|+|L_3|]\right\} \\
            &= \left\{ \deg_{H_t}(a_i) + 2 \colon i \in [|L_3|+1, |L_2|+|L_3|]\right\} \\
            &= \left\{ l_{t-i+1} + 2 \colon i \in [|L_3|+1, |L_2|+|L_3|]\right\} \\
            &= [\left\lfloor \tfrac{t-1}{2}\right\rfloor+2, t].
        \end{aligned}
    \end{align}
    From \eqref{small_degrees_2} and \eqref{large_degrees_2} we have that \ref{con_a2} is satisfied.

    Since $A(t,k)$ in this case does not have half-edges, \ref{con_a3} and \ref{con_a4} are satisfied, and $\ell = 0$.
    Conditions \ref{con_a5} and \ref{con_a6} are satisfied due to Claim~\ref{claim_li_lj} and~\eqref{condition_Li2}.
    
    Hence, $A(t,k)$ satisfies conditions \ref{con_a1}--\ref{con_a6}, and from Lemma~\ref{lemma_A(t,k)} we have that ${\rm lir}(C_{p(t+1)+q(t+2)}^k) = 2$ for each nonnegative integers $p$ and $q$, such that $p+q \geq 2$.

\subsection{Proof of Lemma~\ref{lemma_half_edges}}

    Let $H_t$ be the disconnected almost irregular graph of order $t$, and let $M_t$ be the multiset of degrees of vertices of $H_t$, i.e., 
    $M_t = \{0, 1, \dots, \left\lfloor \tfrac{t-1}{2} \right\rfloor, \left\lfloor \tfrac{t-1}{2} \right\rfloor, \dots, t-2\}$. 
    Moreover, let
    \begin{align*}
        s_1 &= t + \left\lceil \tfrac{t-1}{2} \right\rceil - 2k + 1,\\
        s_2 &= k - \left\lceil \tfrac{t-1}{2} \right\rceil,\\
        s_3 &= 2k - 1 - t,\\
        s_4 &= k-1- \left\lceil \tfrac{t-1}{2} \right\rceil.\\
    \end{align*}
    Note that $s_1$ is nonnegative and $s_2$, $s_3$, and $s_4$ are all positive integers, since $\frac{4k-1}{3} \leq t \leq \frac{8k-5}{5}$.
    Hence, $s_4 - 1$ is nonnegative and it can represent the length of a list in the following.


    Let $L$ and its sublists $L_1$, $L_2$, $L_3$, $L_4$, and $L_5$ be defined in the following way: 
    \begin{align*}
        L = \underbrace{\left\lfloor \tfrac{t-1}{2} \right\rfloor, \dots, \left\lfloor \tfrac{t-1}{2} \right\rfloor + s_1 - 1}_{L_1}, \underbrace{\left\lfloor \tfrac{t-1}{2} \right\rfloor, \dots, \left\lfloor \tfrac{t-1}{2} \right\rfloor - (s_2 - 1)}_{L_2}, \underbrace{\left\lfloor \tfrac{t-1}{2} \right\rfloor + s_1, \dots, t-2}_{L_3}, \\
        \underbrace{\left\lfloor \tfrac{t-1}{2} \right\rfloor - s_2, \dots, \left\lfloor \tfrac{t-1}{2} \right\rfloor - s_2 -(s_4 - 2)}_{L_4}, 0,  \underbrace{\left\lfloor \tfrac{t-1}{2} \right\rfloor - s_2 - (s_4 - 1), \dots, 1}_{L_5}. 
    \end{align*}
    Hence, $L=L_1L_2L_3L_40L_5$, $L_1L_3$ is a list of all integers from $[\left\lfloor\frac{t-1}{2}\right\rfloor, t-2]$ listed in an increasing order, and $L_2L_4L_5$ is a list of all integers from $[1, \left\lfloor\frac{t-1}{2}\right\rfloor]$ listed in a decreasing order.
    Moreover, $|L_1| = |L_5| = s_1$, $|L_2| = s_2$, $|L_3| = s_3$, $|L_4| = s_4-1$, and $|L| = 2s_1 +s_2 + s_3 + s_4= t$.

    By $l_i$ we denote the $i$-th element of $L$, for $i \in [1, |L|]$. 

    \begin{claim}\label{claim_i_minus_j}
        $|i-j| \leq k$ whenever $l_i + l_j \geq t-1$. Moreover, if $l_i \notin L_5$ and $L_j \in L_5$ then $j-i \leq k-1$.
    \end{claim}
    \begin{proof}
        Note that the sum of every two elements of $L_2L_4L_5$ is smaller than $t-1$. 
        Hence, it suffices to prove the claim in the case when $l_i \in L_1L_3$.
    
        First, suppose that $l_i \in L_3$, i.e., $i \in [s_1 + s_2 + 1, s_1 + s_2 + s_3]$.
        Let $j \in [1, |L|]\setminus\{i\}$.
        Clearly, if $j < i$ then $|i-j| \leq s_1+s_2+s_3 - 1 = k-1$.
        Similarly, if $j >i$ then $|i-j| \leq |L| - (s_1 +s_2 +1) = s_1 + s_3 + s_4 - 1 = k-2$.
        Hence, in both cases, $|i-j| \leq k-1$.
    
        Suppose therefore, in the following that $l_i \in L_1$, i.e., $i \in [1, s_1]$.
        Let $j \in [1, |L|] \setminus \{i\}$ be such that $l_i + l_j \geq t-1$. 
        If $l_j \geq \left\lfloor \frac{t-1}{2} \right\rfloor$ then, clearly $|i-j| \leq s_1 +s_2 +s_3 - 1 = k-1$.
        Thus, in the following, assume that $l_j < \left\lfloor \frac{t-1}{2} \right\rfloor$.
    
        Since $l_i \in L_1$, we have that $l_i = \left\lfloor \frac{t-1}{2} \right\rfloor + i - 1$.
        Consider now the largest $j$ such that $l_j < \left\lfloor \frac{t-1}{2} \right\rfloor$ and $l_i + l_j \geq t-1$.
        Clearly, since $L_2L_4L_5$ is the list of elements of $[1, \left\lfloor \frac{t-1}{2}\right\rfloor]$ listed in a decreasing order, we have $l_j = t-1 - l_i = \left\lceil \frac{t-1}{2} \right\rceil - i +1$.
        It follows from the initial assumption $t \leq \frac{8k-5}{5}$ that the sum of the last element of $L_1$, $\left\lfloor \tfrac{t-1}{2} \right\rfloor + s_1 - 1$, and the first element of $L_5$, $\left\lfloor \tfrac{t-1}{2} \right\rfloor - s_2 - s_4 + 1$ is smaller than $t-1$.
        Hence, $j \leq t - s_1 - 1$, i.e., $l_j \in L_2L_4$.
        
        Note that, after $l_j$, all nonnegative integers smaller than $l_j$ are listed in $L$, and there may or may not be some integers greater than $l_j$, namely if $l_j \in L_2$.
        In either case, for $j$ we have $j \leq t - l_j =\left\lfloor\frac{t-1}{2}\right\rfloor + i$, and, subsequently $|i-j| = j-i \leq \left\lfloor\frac{t-1}{2}\right\rfloor \leq k$.
        This completes the proof of the claim.
    \end{proof}

    Consider now the ordering of vertices $a_1, \dots, a_t$ of $H_t$ given by the list of their degrees $L$.
    Two vertices of $H_t$ are adjacent if the sum of their degrees is at least $t-1$. 
    Let $A(t,k)$ be an ordered graph obtained from $H_t$ by 
    \begin{itemize}
        \item adding vertices $a_0$ and $a_{t+1}$,
        \item adding edges $a_0a_i$ for each $i \in [1,s_1+s_2 +s_3]$,
        \item adding edges $a_ia_{t+1}$ for each $i \in [s_1+s_2+1, s_1+s_2+s_3+s_4]$,
        \item adding a half-edge incident to $a_i$ for each $i \in [1, s_1] \cup [s_1+s_2+s_3+s_4+1, t]$.
    \end{itemize}
    We claim that $A(t,k)$ satisfies conditions \ref{con_a1}--\ref{con_a6}.

    Clearly, $\deg_{A(t,k)}(a_0) + \deg_{A(t,k)}(a_{t+1}) = (s_1+s_2+s_3) + (s_3+s_4) \geq t + 1$, hence \ref{con_a1} holds.

    Note that $a_{s_1+s_2+s_3+s_4}$ is a vertex of degree zero in $H_t$ and, thus, it is of degree one in $A(t,k)$, since no half-edge and only one edge, namely $a_{s_1+s_2+s_3+s_4}a_{t+1}$, is incident to $a_{s_1+s_2+s_3+s_4}$ in $A(t,k)$.
    
    Let $L'$ be the list of degrees of vertices of $A(t,k)$, in the order given by the ordering of vertices of $A(t,k)$. 
    Moreover, let $L'_0$, $L'_1$, $L'_2$, $L'_3$, $L'_4$, $L'_5$, and $L'_6$ be the sublists of $L'$ of lengths 1, $s_1$, $s_2$, $s_3$, $s_4-1$, $s_1$, and $1$ respectively, such that
    \begin{align*}
            L' = L'_0L'_1L'_2L'_3L'_41L'_5L'_6.
    \end{align*}
    To show that \ref{con_a2} holds for $A(t,k)$, it is enough to show that $L'_1L'_2L'_3L'_41L'_5$ consists of all integers from $[1,t]$.
    Clearly, $|L'_1L'_2L'_3L'_41L'_5| = |L| = t$.
    Denote by $l'_i$ the $i$-th element of $L'$.
    From the construction of $A(t,k)$ the following observation follows:
    \begin{align}\label{new_elements_of_list}
        l'_i = \left\{\begin{matrix*}[l] l_i + 2 & \text{if } i \in [1, s_1] \cup [s_1+s_2+1, s_1 + s_2+s_3], \\ l_i + 1 & \text{otherwise.} \end{matrix*}\right.
    \end{align}
    Using \eqref{new_elements_of_list} and the fact that $L_1L_3$ consists of all elements of $[\left\lfloor\frac{t-1}{2}\right\rfloor, t-2]$, we get that $L'_1L'_3$ consists of all elements of $[\left\lfloor\frac{t-1}{2}\right\rfloor+2, t]$.
    Similarly $L'_2L'_4L'_5$ consists of all elements of $[2, \left\lfloor\frac{t-1}{2}\right\rfloor+1]$, and, consequently, $L'_1L'_2L'_3L'_41L'_5$ consists of all elements of $[1,t]$.
    Since $|L'| = t$, each element of $[1,t]$ occurs exactly once in $L'$, and, thus, \ref{con_a2} holds for $A(t,k)$.

    It is clear from the definition of $A(t,k)$ that \ref{con_a3} holds. 
    Moreover, \ref{con_a4} holds, since half-edges are incident only to vertices from $\{a_i \colon i \in [1, s_1] \cup [s_1+s_2+s_3+s_4+1, t]\}$, and $t - (s_1+s_2+s_3+s_4) = s_1 \leq \frac{3}{2}t - 2k +1 \leq \frac{t-1}{4} \leq \lfloor \frac{t}{2} \rfloor$ for $t \leq \frac{8k-5}{5}$.

    From Claim~\ref{claim_i_minus_j} and the facts that if $a_0a_i$ and $a_ja_{t+1}$ are edges of $A(t,k)$ then $i \leq s_1+s_2+s_3 = k$ and $j \geq s_1 + s_2 + 1 = t-k + 2$, it follows that \ref{con_a5} and \ref{con_a6} hold.
    Thus, $A(t,k)$ satisfies all conditions \ref{con_a1}--\ref{con_a6}.
    The result then follows directly from Lemma~\ref{lemma_A(t,k)}.

\section{Concluding remarks}

The Local Irregularity Conjecture says that the value of the locally irregular chromatic index of every graph is at most three, with the exception of special cacti from $\mathfrak{T}'$ that are not locally irregular colorable, and a single graph which admits a 4-liec but no 3-liec.
Motivated by the results on graphs of various classes, for which the exact value of the locally irregular chromatic index was determined, and the $(2,2)$-Conjecture, we introduced a problem of determining $\mathcal{D}_\mathrm{lir}(G)$, see Problem~\ref{main_problem}.
This new parameter provides, in some sense, a measure of how far graphs are from admitting a 2-liec.
Namely, we asked for the minimum number of edges which need to be doubled so the resulting multigraph has a 2-liec in which parallel edges are colored the same.
Such a definition relates this new problem closely to $(2,2)$-Conjecture.

However, when considering locally irregular colorings of multigraphs, it seems natural to omit the condition of the same color on parallel edges.
This was, for example, done in~\cite{Grzelec wozniak} and~\cite{Grzelec wozniak2}.
Such an approach then yields the formulation of the new problem:
\begin{problem}\label{problem_red_blue_allowed}
    Let $G$ be a simple connected graph different from $K_2$.
    What is the minimum number of edges $\mathcal{D}_\mathrm{lir}'(G)$ of $G$ which need to be doubled so the resulting multigraph admits a $2$-liec?
\end{problem}
Note that such a formulation of the problem thus allows parallel edges to be of the same color as well as of different colors.
Hence, $\mathcal{D}_\mathrm{lir}'(G) \leq \mathcal{D}_\mathrm{lir}(G)$, and as in the case of $\mathcal{D}_\mathrm{lir}(G)$, $\mathcal{D}_\mathrm{lir}'(G) \geq 1$ whenever $\mathrm{lir}(G) >2$ or $G$ is not locally irregular colorable.
Note that, when Problem~\ref{main_problem} was formulated, on top of $K_2$, also $K_3$ had to be left out from consideration, since it is not possible to double some of the edges of $K_3$ and find a 2-liec of the resulting multigraph in which parallel edges are colored the same.
However, in the case of  Problem~\ref{problem_red_blue_allowed}, such an exception of $K_3$ is not needed, see Figure~\ref{fig:C_3_one _doubling}.
\begin{figure}[h]
    \centering
    \begin{tikzpicture}[scale=0.7]
    \begin{pgfonlayer}{nodelayer}
    \node [style={black_dot}] (0) at (0, 1) {};
    \node [style={black_dot}] (1) at (-1.5, -1) {};
    \node [style={black_dot}] (2) at (1.5, -1) {};
    \end{pgfonlayer}
    \begin{pgfonlayer}{edgelayer}
    \draw [style={blue_edge}] (0) to (1);
    \draw [style={blue_edge}, bend left=15] (1) to (2);
    \draw [style={red_edge}] (0) to (2);
    \draw [style={red_edge}, bend left=15] (2) to (1);
    \end{pgfonlayer}
    \end{tikzpicture}
    \caption{$\mathcal{D}_\mathrm{lir}'(K_3) = 1$.}
    \label{fig:C_3_one _doubling}
\end{figure}
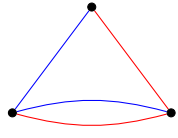
Moreover, if Local Irregularity Conjecture for 2-multigraphs is true then Problem~\ref{problem_red_blue_allowed} has a solution for every connected graph, except for $K_2$.

If $\mathcal{D}_\mathrm{lir}(G) = 1$, we have that $\mathcal{D}_\mathrm{lir}'(G) = 1$ as well, and if $\mathrm{lir}(G) \leq 2$ then $\mathcal{D}_\mathrm{lir}(G) = 0$. 
Hence Problem~\ref{problem_red_blue_allowed} is fully solved (or $\mathcal{D}_\mathrm{lir}'(G)$ is upper bounded by one) for every graph that was shown to have the locally irregular chromatic index at most two, majority of complete graphs, paths, some cycles, trees and split graphs which are not complete graphs. The rest is still widely open. In particular, for complete graphs, there are only five open cases, hence the solution for them might be obtained using the similar method (involving a computer program) that was used to show that $\mathcal{D}_\mathrm{lir}(K_n) \geq 2$ if $n \in \{6, \dots, 10\}$.

Note also that the proof of Corollary~\ref{corollary_D_not_constant} would stay almost unchanged in the case of $\mathcal{D}_\mathrm{lir}'(G)$, since a doubling is needed on the considered pendant triangles of graphs from $\mathfrak{T}$ no matter which of the problems is considered. 
This also gives an insight on providing a bound on $\mathcal{D}_\mathrm{lir}'(G)$ for $G \in \mathfrak{T}^*$ similar to Theorem~\ref{special_cacti}.


\medskip\noindent
{\bf Acknowledgement.} This work was supported by the Slovak Research and Development Agency under the contract No. APVV-23-0191 (Madaras, Onderko, Soták), and partially by the program “Excellence initiative – research university” for the AGH University of Krakow (Grzelec).

\end{document}